\documentclass[a4paper,11pt,reqno]{amsart}

\usepackage{amssymb}
\usepackage{latexsym}
\usepackage{amsmath}
\usepackage{mathrsfs}
\usepackage{euscript}
\usepackage{amsthm}
\usepackage{upgreek}
\usepackage{cite}
\usepackage{xcolor}
\usepackage{tikz}
\usepackage{calc}

\newcommand{\scal}[2]{\langle #1,#2\rangle}

\newcommand{\rr}[1]{\mathbf R^{#1}}
\newcommand{\zz}[1]{\mathbf Z^{#1}}

\newcommand{\cc}[1]{\mathbf C^{#1}}
\newcommand{\nn}[1]{\mathbf N^{#1}}

\newcommand{\nm}[2]{\Vert #1\Vert _{#2}}
\newcommand{\NM}[2]{\left \Vert #1\right \Vert _{#2}}

\newcommand{\sets}[2]{\{ \, #1\, ;\, #2\, \} }

\newcommand{\cdo}{\, \cdot \, }

\newcommand{\supp}{\operatorname{supp}}

\newcommand{\eabs}[1]{\langle #1\rangle}
\newcommand{\BD}{\operatorname{BD}}

\newcommand{\vrum}{\vspace{0.1cm}}

\newcommand{\sfW}{\mathsf{W}}

\newcommand{\maclS}{\mathcal S}

\newcommand{\mascF}{\mathscr F}

\newcommand{\mascP}{\mathscr P}
\newcommand{\mascS }{\mathscr S}

\newcommand{\mabfj}{\boldsymbol j}
\newcommand{\mabfk}{\boldsymbol k}

\numberwithin{equation}{section}          

\newtheorem{thm}{Theorem}
\numberwithin{thm}{section}

\newcommand{\rubrik}{}
\newtheorem{prop}[thm]{Proposition}

\theoremstyle{definition}

\newtheorem{defn}[thm]{Definition}
\newtheorem{example}[thm]{Example}

\theoremstyle{remark}

\newtheorem{rem}[thm]{Remark}              

\author{Nenad Teofanov}

\address{Department of Mathematics and Informatics,
University of Novi Sad, Novi Sad, Serbia}

\email{nenad.teofanov@dmi.uns.ac.rs}

\author{Joachim Toft}

\address{Department of Mathematics,
Linn{\ae}us University, V{\"a}xj{\"o}, Sweden}

\email{joachim.toft@lnu.se}

\title[Multiplications and convolutions on modulation spaces]
{An excursion to multiplications and convolutions on
modulation spaces}



\begin{document}


\begin{abstract}
We give a self-contained introduction to (quasi-)Banach modulation spaces
of ultradistributions,
and review results on boundedness for multiplications and convolutions
for elements in such spaces. Furthermore, we use these results to study the Gabor product.
As an example, we show how it appears in a phase-space formulation of
the nonlinear cubic Schr\"odinger equation.
\end{abstract}

\maketitle

\section{Introduction}\label{sec0}

Modulation spaces were introduced in Feichtinger's seminal
technical report \cite{Fei1}, and prove themselves as
useful family of Banach spaces of tempered distributions
in time-frequency analysis, \cite{BenOko2, Cordero1, Gro2}.
The main purpose of this survey  article is to enlighten some properties of
modulation spaces
in a rather  self-contained manner. In contrast to the most common
situation,
our analysis includes both quasi-Banach and Banach modulation
spaces within the framework of ultradifferentiable functions and
ultradistributions of Gelfand--Shilov type. For that reason we
collect necessary background material
in a rather detailed preliminary section.

Motivated by recent applications of modulation spaces in the
context of nonlinear harmonic analysis and its applications,
cf.
\cite{BenGraGroOko, BenGroOkoRog,
BenOko2,DPT2022,FeiNar, OhWang,OhWang2,Teofanov4,Toft26}
we focus our attention to boundedness for multiplications and
convolutions
for elements in such spaces. The basic results in that direction go
back to the original contribution \cite{Fei1}, and were thereafter
reconsidered by many authors in different contexts. Let us give a
brief, and unavoidably incomplete account on
the related results.

\par

In Section \ref{sec2} we formulate in
Theorems \ref{Thm:MultMod1} and \ref{Thm:ConvMod1}
bilinear versions of more general multiplication
and convolution results in \cite[Section 3]{Toft26}.
The contents of Theorems \ref{Thm:MultMod1} and
\ref{Thm:ConvMod1} in the unweighted case for
modulation spaces $M^{p,q}$ can be summarized as follows.

\par

\begin{prop}\label{Prop:IntroMultConvMod}
Let $p_j,q_j\in (0,\infty ]$, $j=0,1,2$,
$$
\theta _1=\max \left (1 ,\frac 1{p_0},\frac 1{q_1},\frac 1{q_2} \right )
\quad \text{and}\quad
\theta _2=\max \left (1 ,\frac 1{p_1},\frac 1{p_2} \right ).
$$
Then
\begin{alignat*}{3}
M^{p_1,q_1}\cdot M^{p_2,q_2} &\subseteq M^{p_0,q_0},&
\qquad
\frac 1{p_1}+\frac 1{p_2} &=\frac 1{p_0}, &
\quad
\frac 1{q_1}+\frac 1{q_2} &= \theta _1 + \frac 1{q_0},
\\[1ex]
M^{p_1,q_1}* M^{p_2,q_2} &\subseteq M^{p_0,q_0}, &
\qquad
\frac 1{p_1}+\frac 1{p_2} &= \theta _2+\frac 1{p_0}, &
\quad
\frac 1{q_1}+\frac 1{q_2} &= \frac 1{q_0}.
\end{alignat*}
\end{prop}

\par

The general multiplication and convolution properties in Section \ref{sec2} also overlap with
results by Bastianoni, Cordero and Nicola in \cite{BaCoNi},
by Bastianoni and Teofanov in \cite{BaTe},
and by Guo, Chen, Fan and Zhao in \cite{GuChFaZh}.

\par

The multiplication
relation in Proposition \ref{Prop:IntroMultConvMod} for $p_j,q_j\ge 1$
was obtained already in \cite{Fei1} by Feichtinger. It is also obvious that
the convolution relation was well-known since then
(though a first formal proof of this relation seems to be given  in
\cite{Toft3}). In general, these convolution and multiplication
properties follow the rules
\begin{alignat*}{4}
\ell ^{p_1}&*\ell ^{p_2}\subseteq \ell ^{p_0}, & \quad
\ell ^{q_1} &\cdot \ell ^{q_2}\subseteq \ell ^{q_0} &
\quad &\Rightarrow & \quad
M^{p_1,q_1}&*M^{p_2,q_2}\subseteq M^{p_0,q_0}
\intertext{and}
\ell ^{p_1} &\cdot \ell ^{p_2}\subseteq \ell ^{p_0}, & \quad
\ell ^{q_1} &* \ell ^{q_2}\subseteq \ell ^{q_0} &
\quad &\Rightarrow & \quad
M^{p_1,q_1} &\cdot M^{p_2,q_2}\subseteq M^{p_0,q_0},
\end{alignat*}
which goes back to \cite{Fei1} in the Banach space case and to
\cite{GaSa} in the quasi-Banach case. See also \cite{FeiGro1}
and \cite{Rau2}
for extensions of these relations to more general Banach function
spaces and quasi-Banach function spaces, respectively.

\par

In Section \ref{sec2} we basically review some results from \cite{Toft26}.
To make this survey self-contained we give the proof of  Theorem \ref{Thm:ConvMod1}
in unweighted case. In contrast to \cite{GuChFaZh}, we do not deduce any sharpness for our results.

\par

To show Proposition \ref{Prop:IntroMultConvMod} in the quasi-Banach setting,
apart from the usual use of  of H{\"o}lder's and
Young's inequalities, additional
arguments are needed. In our situation we
discretize the situations in similar ways as in \cite{BaCoNi} by using
Gabor analysis for modulation spaces, and then apply some
further arguments, valid in non-convex analysis. This
approach is slightly different compared to what is used in
\cite{GuChFaZh} which follows the discretization technique
introduced in \cite{WaHu}, and which has some traces of Gabor analysis.

\par

We refer to \cite{Toft26} for a detailed discussion on the uniqueness of multiplications
and convolutions in Proposition \ref{Prop:IntroMultConvMod}.

\par

In Section \ref{Sec:3} we apply  the results from previous parts in the
framework of the so called Gabor product. It is introduced in \cite{DPT2022}
in order to derive a phase space analogue to the usual convolution identity for the Fourier transform.
The main motivation is to  use such kind of products in a phase-space formulation of certain nonlinear equations.
As noticed in \cite{DPT2022}, among other interesting characteristics of
phase-space representations, the initial value problem in phase-space may be well-posed for more general initial distributions. This means that the phase-space formulation could contain solutions other than the standard ones. We refer to
\cite{Dias1,Dias2,Dias3}, where the phase-space extensions are explored in different contexts.
Here we illustrate this approach by considering the nonlinear cubic Schr\"odinger equation, which appear for example in
in Bose-Einstein condensate theory \cite{Kevrekidis}.
We also  refer to \cite[Chapter 7]{BenOko2} for an overview of results related to well-posedness of the  nonlinear  Schr\"odinger equations in the framework of modulation spaces, see also \cite{BenOko1, OhWang, OhWang2}.

\par

\section*{Acknowledgement}

\par

The work of N. Teofanov is partially supported  by TIFREFUS
Project DS 15, and MPNTR of Serbia Grant
No. 451--03--68/2022--14/200125.
Joachim Toft was supported by Vetenskapsr{\aa}det
(Swedish Science Council) within the project 2019-04890.

\par

\section{Preliminaries}\label{sec1}

\par

In this section we give an exposition of
background material related to the definition and basic properties of modulation spaces.
Thus we recall some facts on the short-time Fourier transform and related projections,
the (Fourier invariant) Gelfand-Shilov spaces,
weight functions, and mixed-norm spaces of Lebesgue type. We also recall
convolution and multiplication in weighted Lebesgue sequence spaces.

\par

\subsection{The short-time Fourier transform}\label{subsec1.1}

\par

In what follows we let $\mathscr F$ be the
Fourier transform which takes the form
$$
(\mathscr Ff)(\xi )= \widehat f(\xi ) \equiv (2\pi )^{-\frac
d2}\int _{\rr
{d}} f(x)e^{-i\scal  x\xi }\, dx
$$
when $f\in L^1(\rr d)$. Here $\scal \cdo \cdo$ denotes the usual
scalar product
on $\rr d$. The same notation is used for the usual dual form
between test functions and corresponding
(ultra-)distributions. We recall that map $\mathscr F$ extends
uniquely to a homeomorphism on the space of tempered distributions
$\mascS '(\rr d)$,
to a unitary operator on $L^2(\rr d)$ and restricts
to  a homeomorphism on the Schwartz space of smooth rapidly
decreasing functions $\mathscr S(\rr d)$, cf.
\eqref{Eq:SchwartzSpDef}.
We also observe with our choice of the Fourier
transform, the usual
convolution identity for the Fourier transform
takes the forms
\begin{equation}\label{Eq:FourTransfConv}
\mascF (f\cdot g)
=
(2\pi )^{-\frac d2}\widehat f *\widehat g
\quad \text{and}\quad
\mascF (f* g)
=
(2\pi )^{\frac d2}\widehat f \cdot \widehat g
\end{equation}
when $f,g\in \mascS (\rr d)$.

\par

In several situations it is convenient to use a localized version
of the Fourier transform, called the short-time Fourier transform, STFT for short.
The short-time Fourier transform of
$f\in \mascS '(\rr d)$ with respect to the fixed \emph{window function}
$\phi \in \mascS (\rr d)$ is defined by
\begin{align}
(V_\phi f)(x,\xi ) &\equiv (2\pi )^{-\frac d2}
(f,\phi (\cdo -x)e^{i\scal \cdo \xi})_{L^2}.
\label{Eq:STFTDef}
\intertext{Here $(\cdo ,\cdo )_{L^2}$ denotes the unique continuous
extension of the inner product on $L^2(\rr d)$ restricted to
$\mascS (\rr d)$ into a continuous map from $\mascS '(\rr d)
\times \mascS (\rr d)$ to $\mathbf C$.
\newline
\indent
We observe that using certain properties for tensor products
of distributions, }
(V_\phi f)(x,\xi ) &= \mascF (f\cdot \overline {\phi (\cdo -x)})(\xi ).
\tag*{(\ref{Eq:STFTDef})$'$}
\intertext{(cf. \cite{Ho1,Toft22}). If in addition
$f\in L^p(\rr d)$
for some $p\in [1,\infty ]$, then}
(V_\phi f)(x,\xi ) &= (2\pi )^{-\frac d2}\int _{\rr d}f(y)
\overline {\phi (y-x)}e^{-i\scal y\xi}\, dy.
\tag*{(\ref{Eq:STFTDef})$''$}
\end{align}
We observe that the domain of $V_\phi$
is $\mascS '(\rr d)$.
The images are contained
in $C^\infty (\rr {2d})$, the set of smooth functions
defined on the phase space $\rr d\times \rr d\simeq \rr {2d}$.

\par

The short-time Fourier transform appears in different contexts
and under different names. In quantum mechanics it is rather common
to it the \emph{coherent state transform}
(see e.{\,}g. \cite{LieSol}). It is also closely related to the
so-called Wigner distribution or radar ambiguity function (see e.{\,}g.
\cite{Lieb1}). In time-frequency analysis, it is also sometimes called a
\emph{Voice transform}.

\par

The main idea with the design of short-time Fourier transform is
to get the Fourier content, or the frequency resolution of localized
functions and distributions. Roughly speaking, short-time
Fourier transform give a simultaneous information both on
functions or distributions themselves as well as their
Fourier transforms in the sense that the map
$$
x\mapsto V_\phi f(x,\xi )
$$
resembles on $f(x)$, while the map
$$
\xi \mapsto V_\phi f(x,\xi )
$$
resembles on $\widehat f(\xi )$.

\par

As for the ordinary Fourier transform, there are several mapping
properties which hold true for the short-time Fourier transform.
An elegant way to approach such mapping in the framework of
distributions, we may follow ideas given in \cite{Fol} by Folland.

\par

In fact, let $T$ is the semi-conjugated tensor map
\begin{align}
T(f,\phi ) &= f\otimes \overline \phi ,
\intertext{$U$ be the linear pullback} 
(UF)(x,y) &= U(y,y-x)
\intertext{and $\mascF _2$ be the partial Fourier transform given by}
(\mascF _2F)(x,\xi ) &= (2\pi )^{-\frac d2}\int _{\rr d}F(x,y)e^{-i\scal y\xi}\, dy .
\end{align}
Then
\begin{equation}\label{Eq:STFTDecomp}
V_\phi f = (\mascF _2 \circ U\circ T)(f,\phi ) ,
\end{equation}
when $f,\phi \in \mascS (\rr d)$.

\par

We observe that the mappings
\begin{alignat}{2}
T : \mascS (\rr d) \times \mascS (\rr d) &\to \mascS (\rr {2d}), &
\quad
U,\mascF _2 : \mascS (\rr {2d}) &\to \mascS (\rr {2d})
\intertext{are continuous and uniquely extendable to continuous
mappings}
T : \mascS '(\rr d) \times \mascS '(\rr d) &\to \mascS '(\rr {2d}), &
\quad
U,\mascF _2 : \mascS '(\rr {2d}) &\to \mascS '(\rr {2d}),
\intertext{which in turn restricts to isometric mappings}
T : L^2(\rr d) \times L^2(\rr d) &\to L^2(\rr {2d}), &
\quad
U,\mascF _2 : L^2(\rr {2d}) &\to L^2(\rr {2d}).
\end{alignat}
Here that $T$ is isometric means that
$$
\nm {T(f,\phi )}{L^2(\rr {2d})} = \nm f{L^2(\rr d)}\nm \phi {L^2(\rr d)}.
$$

\par

It is now natural to define $V_\phi f$ as the right-hand side of
\eqref{Eq:STFTDecomp} when $f,\phi \in \mascS '(\rr d)$, in which
$V_\phi f$ is well-defined as an element in $\mascS '(\rr {2d})$.

\par


\begin{prop}\label{Prop:ExtSTFTSchwartz}
The map
\begin{alignat}{4}
(f,\phi ) &\mapsto V_\phi f &  &:\,  &
&\mascS (\rr d) \times \mascS (\rr d) & &\to \mascS (\rr {2d})
\label{Eq:STFTSchwartz}
\intertext{is continuous, which extends uniquely to a continuous map}
(f,\phi ) &\mapsto V_\phi f & &:\,  &
&\mascS '(\rr d) \times \mascS '(\rr d) & &\to \mascS '(\rr {2d}),
\label{Eq:STFTTempDist}
\intertext{which in turn restricts to an isometric map}
(f,\phi ) &\mapsto V_\phi f & &:\,  &
&L^2(\rr d) \times L^2(\rr d) & &\to L^2(\rr {2d}).
\label{Eq:STFTL2}
\end{alignat}
\end{prop}

\par

%

If $\phi \in \mascS (\rr d)$ and $f\in \mascS '(\rr d)$, then
\eqref{Eq:STFTTempDist} shows that $V_\phi f\in \mascS '(\rr {2d})$.
On the other hand, it is easy to see that the right-hand side
of \eqref{Eq:STFTDef} defines a smooth function. Consequently
beside \eqref{Eq:STFTTempDist} and \eqref{Eq:STFTSchwartz},
we also have the continuous map
\begin{equation}\label{Eq:STFTTempDistSchwartz}
(f,\phi ) \mapsto V_\phi f   :\,
\mascS '(\rr d) \times \mascS (\rr d) \to \mascS '(\rr {2d})
\cap C^\infty (\rr {2d}).
\end{equation}

\par

For short-time Fourier transform, the Parseval identity
is replaced by the so-called Moyal identity, also known as the {\em orthogonality relation}
given by
\begin{equation}\label{Eq:Moyal}
(V_{\phi}f,V_{\psi}g)_{L^2(\rr {2d})}
= (\psi ,\phi )_{L^2(\rr d)}(f,g)_{L^2(\rr d)},
\end{equation}
when $f,g,\phi ,\psi \in \mascS (\rr d)$. The identity \eqref{Eq:Moyal}
is obtained by rewriting the short-time Fourier transforms by
\eqref{Eq:STFTDef}$'$ and then applying the Parseval identity
in suitable ways. We observe that the right-hand side
makes sense also when $f$, $g$, $\phi$ and $\psi$ belong to
other spaces than $\mascS (\rr d)$. For example we may let
\begin{equation}\label{Eq:STFTScalProdMap}
\begin{aligned}
(f,g,\phi ,\psi ) &\in \mascS '(\rr d)\times \mascS (\rr d)\times \mascS (\rr d)
\times \mascS '(\rr d),
\\[1ex]
(f,g,\phi ,\psi ) &\in \mascS (\rr d)\times \mascS '(\rr d)\times \mascS '(\rr d)
\times \mascS (\rr d),
\\[1ex]
(f,g,\phi ,\psi ) &\in \mascS '(\rr d)\times \mascS (\rr d)\times L^q(\rr d)
\times L^{q'}(\rr d)
\\[1ex]
\text{or}\qquad
(f,g,\phi ,\psi ) &\in L^p(\rr d)\times L^{p'}(\rr d)\times L^q(\rr d)
\times L^{q'}(\rr d),
\end{aligned}
\end{equation}
when $p,p',q,q'\in [1,\infty]$ satisfy
$$
\frac 1p+\frac 1{p'}=\frac 1q+\frac 1{q'}=1.
$$

\par

By Moyal's identity \eqref{Eq:Moyal} it follows that if
$\phi \in \mascS (\rr d)\setminus 0$, then the identity operator
on $\mascS '(\rr d)$ is given by
\begin{equation}\label{Eq:IdentSTFTAdj}
\operatorname{Id}
=
\left (\nm \phi{L^2}^{-2}\right ) \cdot V_\phi ^*\circ V_\phi ,
\end{equation}
provided suitable mapping properties
of the ($L^2$-)adjoint $V_\phi ^*$ of $V_\phi$ can be established.
Obviously, $V_\phi ^*$ fullfils
\begin{equation}\label{Eq:STFTAdjoint}
(V_\phi ^*F,g) _{L^2(\rr d)} = (F,V_\phi g) _{L^2(\rr {2d})}
\end{equation}
when $F\in \mascS (\rr {2d})$ and $g\in \mascS (\rr d)$.

\par

By expressing the scalar product and the short-time Fourier
transform in terms of integrals in \eqref{Eq:STFTAdjoint},
it follows by straight-forward manipulations that
the adjoint in \eqref{Eq:STFTAdjoint} is given by
\begin{equation}\label{Eq:STFTAdjointFormula}
(V_\phi ^*F)(x) = (2\pi )^{-\frac d2}
\iint _{\rr {2d}} F(y,\eta )\phi (x-y)e^{i\scal x\eta}\, dyd\eta ,
\end{equation}
when $F\in \mascS (\rr {2d})$. We may now use mapping properties
like \eqref{Eq:STFTTempDist}--\eqref{Eq:STFTL2} to extend
the definition of $V_\phi ^*F$ when $F$ and $\phi$ belong to
various classes of function and distribution spaces. For example,
by \eqref{Eq:STFTTempDist},
\eqref{Eq:STFTSchwartz} and \eqref{Eq:STFTL2},
it follows that the map
$$
(F,g)\mapsto (F,V_\phi g) _{L^2(\rr {2d})}
$$
defines a sesqui-linear form on
$\mascS (\rr {2d})\times \mascS '(\rr d)$,
$\mascS '(\rr {2d})\times \mascS (\rr d)$ and on
$L^2(\rr {2d})\times L^2(\rr d)$. This implies that
if $\phi \in \mascS (\rr d)$, then
$V_\phi ^*$ in \eqref{Eq:STFTAdjoint}  is continuous
from $\mascS (\rr {2d})$ to $\mascS (\rr d)$ which is
uniquely extendable to a continuous map
$\mascS '(\rr {2d})$ to $\mascS '(\rr d)$, and to
$L^2(\rr {2d})$ to $L^2(\rr d)$. That is, the mappings
\begin{equation}\label{Eq:STFTAdjCont}
\begin{alignedat}{3}
V_\phi ^* :   \mascS (\rr {2d}) \to \mascS (\rr d),
\qquad
&V_\phi ^* & &: & \mascS '(\rr {2d}) &\to \mascS '(\rr d)
\\[1ex]
\text{and}\qquad
&V_\phi ^* & &: & L^2(\rr {2d}) &\to L^2(\rr d)
\end{alignedat}
\end{equation}
are continuous.

\par

\subsection{STFT projections and a suitable twisted
convolution}\label{subsec1.2}

\par

If $\phi \in \mascS (\rr d)$ satisfies $\nm \phi{L^2}=1$, then
\eqref{Eq:IdentSTFTAdj} shows that $V_\phi ^*\circ V_\phi$
is the identity operator on $\mascS '(\rr d)$. If we swap the
order of this composition we get certain types of projections.
In fact, for any $\phi \in \mascS (\rr d)\setminus 0$, let $P_\phi$
be the operator given by
\begin{equation}\label{Eq:ProjphiDef}
P_\phi \equiv \nm \phi {L^2}^{-2}\cdot V_\phi \circ V_\phi ^*.
\end{equation}
We observe that $P_\phi$ is continuous on $\mascS (\rr {2d})$,
$L^2(\rr {2d})$ and $\mascS '(\rr {2d})$ due to the mapping properties
for $V_\phi$ and $V_\phi ^*$ above.

\par

It is clear that $P_\phi ^*=P_\phi$, i.{\,}e. $P_\phi$ is self-adjoint. Furthermore,
$P_\phi$ is an involution:
$$
P_\phi ^2 = \nm \phi {L^2}^{-2}\cdot
V_\phi \circ
\Big (
\underset{\text{The identity operator}}
{\underbrace{\nm \phi {L^2}^{-2}\cdot V_\phi ^*\circ V_\phi}}
\Big )
\circ V_\phi ^*
=
\nm \phi {L^2}^{-2}\cdot V_\phi \circ V_\phi ^* =P_\phi .
$$
Hence,
\begin{equation}\label{Eq:ProjphiRule}
P_\phi ^* = P_\phi
\quad \text{and}\quad P_\phi ^2=P_\phi ,
\end{equation}
which shows that $P_\phi$ is an orthonormal projection.

\par

The ranks of $P_\phi$ are given by
\begin{equation}\label{Eq:STFTProjMaps}
\begin{aligned}
P_\phi (\mascS (\rr {2d})) = V_\phi (\mascS (\rr d)),
\qquad
P_\phi (L^2 (\rr {2d})) &= V_\phi (L^2 (\rr d)),
\\[1ex]
\text{and}\qquad
P_\phi (\mascS '(\rr {2d})) &= V_\phi (\mascS '(\rr d)).
\end{aligned}
\end{equation}
In fact, if $F\in \mascS '(\rr {2d})$, then
$$
P_\phi F = V_\phi f,
$$
where $f=\nm {\phi}{L^2}^{-2}V_\phi ^*F\in \mascS '(\rr d)$. This shows that
$P_\phi (\mascS '(\rr {2d})) \subseteq V_\phi (\mascS '(\rr d))$.
On the other hand, if $f\in \mascS '(\rr d)$ and $F=V_\phi f$, then
$$
P_\phi F =   \Big (V_\phi \circ \Big ( \nm \phi {L^2}^{-2} \cdot
V_\phi ^* \circ V_\phi \Big ) \Big )f = V_\phi f,
$$
which shows that any element in $V_\phi (\mascS '(\rr d))$
equals to an element in $P_\phi (\mascS '(\rr {2d}))$, i.e. $P_\phi (\mascS '(\rr {2d})) = V_\phi (\mascS '(\rr d))$. This gives
the last identity in \eqref{Eq:STFTProjMaps}. In the same way, the
first two identities are obtained.

\par

\begin{rem}
Let $F\in \mascS'(\rr {2d})$. Then it follows from
the last identity in \eqref{Eq:STFTProjMaps} that
$F=V_{\phi}f$ for some $f\in \mascS'(\rr {d})$, if and only if
\begin{equation}\label{Eq:TwistedProj}
F= P_\phi F.
\end{equation}
Furthermore, if \eqref{Eq:TwistedProj} holds, then $F=V_{\phi}f$ with
\begin{equation}\label{Eq:TwistedProj2}
f=(\nm \phi{L^2}^{-2})\cdot V_\phi ^*F .
\end{equation}
%
%
%
%
\end{rem}

\par

There is a twisted convolution which is linked to
the projection in \eqref{Eq:ProjphiDef}. In fact, if $F\in \mascS (\rr {2d})$
and $\phi \in \mascS (\rr d)\setminus 0$, then it follows
by expanding the integrals for $V_\phi$ and $V_\phi ^*$ in
\eqref{Eq:ProjphiDef}, and performing some straight-forward
manipulations that
\begin{equation}\label{Eq:ProjOpTwistedConv}
P_\phi F = \nm \phi {L^2}^{-2}\cdot V_\phi \phi *_V F ,
\qquad F\in \mascS '(\rr {2d}),
\end{equation}
where the
\emph{twisted convolution} $*_V$ is defined by
\begin{align}
(F*_VG)(x,\xi )
&=
(2\pi )^{-\frac d2}
\iint _{\rr {2d}} \! \!
F(x-y,\xi -\eta )G(y,\eta )e^{-i\scal{y}{\xi -\eta}}\, dyd\eta .
\notag
\\[1ex]
&=
(2\pi )^{-\frac d2}
\iint _{\rr {2d}} \! \!
F(y,\eta )G(x-y,\xi - \eta )e^{-i\scal{x-y}\eta}\, dyd\eta ,
\label{Eq:TwistConvDef}
\end{align}
when $F,G\in \mascS (\rr {2d})$. We observe that the definition of
$*_V$ is uniquely extendable in different ways. For example, Young's
inequality for ordinary convolution also holds for the twisted
convolution. Moreover, the map $(F,G)\mapsto F*_VG$
extends uniquely to continuous mappings from $\mascS (\rr {2d}) \times
\mascS '(\rr {2d})$ or $\mascS '(\rr {2d})\times \mascS (\rr {2d})$
to $\mascS '(\rr {2d})$. By straight-forward computations
it follows that
\begin{equation}\label{Eq:TwistedConvAsoc}
(F*_VG)*_VH = F*_V(G*_VH),
\end{equation}
when $F,H\in \mascS (\rr {2d})$ and $G\in \mascS '(\rr {2d})$,
or $F,H\in \mascS '(\rr {2d})$ and $G\in \mascS (\rr {2d})$

\par

Let $f\in \mascS'(\rr d)$ and $\phi _j\in \mascS(\rr d)$,
$j=1,2,3$.
By straight-forward applications of Parseval's formula it follows
that
\begin{equation}\label{Eq:STFTWindTrans}
\big ( (V_{\phi _2}\phi _3) *_V(V_{\phi _1}f)  \big ) (x,\xi )
=
(\phi _3,\phi _1)_{L^2} \cdot (V_{\phi _2}f)(x,\xi ),
\end{equation}
which is some sort of reproducing kernel
of short-time Fourier transforms in the background of
$*_V$.

\par

\subsection{Gelfand-Shilov spaces}\label{subsec1.2}

\par

Before defining the Gelfand-Shilov spaces, we recall that
the Schwartz space $\mascS (\rr d)$ consists of all
(complex-valued) smooth functions $f\in C^\infty (\rr d)$
such that
\begin{equation}\label{Eq:SchwartzSpDef}
\sup _{x\in \rr d}\big ( |x^\beta \partial ^\alpha
f(x)| \big ) \le C_{\alpha ,\beta},
\end{equation}
for some constants $C_{\alpha ,\beta}>0$, which only depend on
the multi-indices $\alpha ,\beta \in \nn d$. The Schwartz space
possess several convenient properties. and is heavily used in
mathematics, science and technology. For example, the Schwartz
space is invariant under Fourier transformation. By duality the
same holds true for its ($L^2$-)dual $\mascS '(\rr d)$, the set of
tempered distributions on $\rr d$.

\par

On the other hand, we observe that there are no conditions on the growths
of the constants $C_{\alpha ,\beta}$ with respect to
$\alpha ,\beta \in \nn d$. This implies that in the context of the spaces
$\mascS (\rr d)$ and $\mascS '(\rr d)$, it is almost impossible to investigate
important properties like analyticity or related regularity properties
which are stronger than pure smoothness. In order for investigating
such stronger regularity properties, we need to modify $\mascS (\rr d)$
and the estimate \eqref{Eq:SchwartzSpDef} by imposing  suitable growth conditions on the
constants $C_{\alpha ,\beta}$. This leads to the definition of
Gelfand-Shilov spaces, \cite{GeSh, Pil1}.

\par

We only discuss Fourier invariant
Gelfand-Shilov spaces and their properties.
Let $0<s\in \mathbf R$ be fixed. We have two different
types of Gelfand-Shilov spaces. The Gelfand-Shilov
space $\mathcal S_{s}(\rr d)$
of Roumieu type with parameter $s$
consists of all $f\in C^\infty (\rr d)$ such that
\begin{equation}\label{Eq:GeShSpDef}
\sup _{x\in \rr d}\big ( |x^\beta \partial ^\alpha
f(x)| \big ) \le Ch^{|\alpha  + \beta |}(\alpha ! \beta !)^s,
\end{equation}
for some constants $C,h>0$. In the same way, the Gelfand-Shilov
space $\Sigma _{s}(\rr d)$ of Beurling type with parameter $s$
consists of all $f\in C^\infty (\rr d)$ such that for every $h>0$, there
is a constant $C=C_h>0$ such that \eqref{Eq:GeShSpDef}. Hence,
in comparison with the definition of Schwartz functions, we have limited
ourself to  constants $C_{\alpha ,\beta}$ in
\eqref{Eq:SchwartzSpDef} which are not allowed to grow faster than those
of the form
$$
Ch^{|\alpha  + \beta |}(\alpha ! \beta !)^s
$$
when dealing with Gelfand-Shilov spaces.

\par


It can be proved that $\maclS _s(\rr d)$ and $\Sigma _t(\rr d)$
are dense in $\mascS (\rr d)$ when $s\ge \frac 12$ and $t>\frac 12$. We call such $s$ and $t$ admissible.
On the other hand, for the other choices of $s$ and $t$
we have
$$
\maclS _s(\rr d)=\Sigma _t(\rr d)=\{0\} ,
\quad \text{when}\quad
s<\frac 12 ,\ t\le \frac 12.
$$

\par

One has that $\maclS _1(\rr d)$ consists of real analytic functions,
and that $\Sigma _1(\rr d)$ consists of smooth functions on
$\rr d$ which are extendable to entire functions on $\cc d$.
The topologies of  $\mathcal S_{s}(\rr d)$ and  $\Sigma _{s}(\rr d)$
are defined by the semi-norms
\begin{equation}\label{gfseminorm}
\nm f{\mathcal S_{s,h}}\equiv \sup \frac {|x^\beta \partial ^\alpha
f(x)|}{h^{|\alpha  + \beta |}(\alpha ! \beta !)^s}.
\end{equation}
Here the supremum should be taken
over all $\alpha ,\beta \in \mathbf N^d$ and $x\in \rr d$. We equip
$\mathcal S_{s}(\rr d)$ and $\Sigma _{s}(\rr d)$ by the canonical inductive limit
topology and projective limit topology, respectively, with respect to $h>0$, which are
induced by the semi-norms in \eqref{gfseminorm}.

\medspace

Let $\mathcal S_{s,h}(\rr d)$ be the Banach space which
consists of all $f\in C^\infty (\rr d)$ such that $\nm f{\mathcal S_{s,h}}$
in \eqref{gfseminorm} is finite, and let $\mathcal S_{s,h}'(\rr d)$ be
the ($L^2$-)dual of $\mathcal S_{s,h}(\rr d)$. If $s\ge \frac 12$, then
the \emph{Gelfand-Shilov distribution space} $\mathcal S_{s}'(\rr d)$
of \emph{Roumieu type} is the projective limit of
$\mathcal S_{s,h}'(\rr d)$ with respect to $h>0$. If instead
$s> \frac 12$, then
the \emph{Gelfand-Shilov distribution space} $\Sigma _{s}'(\rr d)$
of \emph{Beurling type} is the inductive limit of
$\mathcal S_{s,h}'(\rr d)$ with respect to $h>0$. Consequently,
for admissible $s$ we have
$$
\mathcal S_{s}'(\rr d)
=
\bigcap _{h>0} \mathcal S_{s,h}'(\rr d)
\quad \text{and}\quad
\Sigma _{s}'(\rr d)
=
\bigcup _{h>0}  \mathcal S_{s,h}'(\rr d).
$$
It can be proved that $\mathcal S_{s}'(\rr d)$ and $\Sigma _{s}'(\rr d)$
are the (strong) duals to $\mathcal S_{s}(\rr d)$ and $\Sigma _{s}(\rr d)$,
respectively.

\par

We have the following embeddings and density properties for
Gelfand-Shilov and Schwartz spaces we get
\begin{equation}\label{GSembeddings}
\begin{alignedat}{3}
\maclS _{s}(\rr d)
&\hookrightarrow &
\Sigma _{t}(\rr d)
&\hookrightarrow &
\maclS _t(\rr d)
&\hookrightarrow
\mascS (\rr d),
\\[1ex]
\mascS '(\rr d)
&\hookrightarrow &
\maclS _t' (\rr d)
&\hookrightarrow &
\Sigma _{t}'(\rr d)
&\hookrightarrow
\maclS _{s}'(\rr d),
\qquad t>s\ge \frac 12,
\end{alignedat}
\end{equation}
with dense embeddings.
Here $A\hookrightarrow B$ means that
the topological spaces $A$ and $B$ satisfy $A\subseteq B$ with
continuous embeddings.

\medspace

The Fourier transform possess convenient mapping properties
on Gelfand-Shilov spaces and their distribution spaces.
In fact, the Fourier transform extends
uniquely to homeomorphisms on $\mascS '(\rr d)$,
$\maclS _s'(\rr d)$ and on $\Sigma _s'(\rr d)$ for admissible $s$.
Furthermore,  $\mascF$ restricts to homeomorphisms on
$\maclS _s(\rr d)$ and on $\Sigma _s(\rr d)$.

 \par

One of the most important characterizations of
Gelfand-Shilov spaces is performed in terms of
estimates of the functions and their Fourier
 transforms. More precisely, in \cite{ChuChuKim, Eij} it is proved that
 if $f\in \mascS '(\rr d)$ and $s>0$, then $f\in \maclS _s(\rr d)$
 ($f\in \Sigma _s(\rr d)$), if and only if
 \begin{equation}\label{Eq:GSFtransfChar}
 |f(x)|\lesssim e^{-r|x|^{\frac 1s}}
 \quad \text{and}\quad
 |\widehat f(\xi )|\lesssim e^{-r|\xi |^{\frac 1s}},
 \end{equation}
 for some $r>0$ (for every $r>0$).
 Here $g_1 \lesssim g_2$ means that $g_1(\theta ) \le c \cdot  g_2(\theta )$
 holds uniformly for all $\theta$
 in the intersection of the domains of $g_1$ and $g_2$
 for some constant $c>0$, and we
 write $g_1\asymp g_2$
 when $g_1\lesssim g_2 \lesssim g_1$.

\par

The analysis in \cite{ChuChuKim, Eij} can also be applied on
the Schwartz space, from which it follows that an element
$f\in \mascS '(\rr d)$ belongs to $\mascS (\rr d)$, if and only if
\begin{equation}\label{Eq:SchwartzFtransfChar}
|f(x)|\lesssim \eabs x^{-N}
\quad \text{and}\quad
|\widehat f(\xi )|\lesssim \eabs \xi ^{-N},
\end{equation}
for every $N\ge 0$. Here and in what follows we let
$$
\eabs x = (1+|x|^2)^{\frac 12}.
$$

\par

\begin{rem}\label{Rem:SchwartzToGS}
Several properties in Subsections \ref{subsec1.1} -- \ref{subsec1.2}
in the background of  $\mascS (\rr d)$ and $\mascS '(\rr d)$
also hold for the Gelfand-Shilov spaces and their distribution spaces.
Let $s\ge \frac 12$. By similar arguments which lead to Proposition
\ref{Prop:ExtSTFTSchwartz} and \eqref{Eq:STFTTempDistSchwartz},
it follows that
\begin{alignat}{4}
(f,\phi ) &\mapsto V_\phi f &  &:\,  &
&\maclS _s(\rr d) \times \maclS _s(\rr d) & &\to \maclS _s(\rr {2d})
\label{Eq:STFTGS}
\intertext{is continuous, which extends uniquely to continuous mappings}
(f,\phi ) &\mapsto V_\phi f & &:\,  &
&\maclS _s'(\rr d) \times \maclS _s(\rr d) & &\to
\maclS _s'(\rr {2d})\cap C^\infty (\rr {2d})
\label{Eq:STFTGSDistSoft}
\intertext{and}
(f,\phi ) &\mapsto V_\phi f & &:\,  &
&\maclS _s'(\rr d) \times \maclS _s'(\rr d) & &\to \maclS _s'(\rr {2d}).
\label{Eq:STFTGSDist}
\end{alignat}

\par

It follows that \eqref{Eq:Moyal}
makes sense after each $\mascS$ in
\eqref{Eq:STFTScalProdMap} are replaced
by $\maclS _s$. Let $\phi \in \maclS _s(\rr d)\setminus 0$
be fixed. Then by similar arguments which lead to \eqref{Eq:STFTAdjCont}
give that the mappings
\begin{equation}\tag*{(\ref{Eq:STFTAdjCont})$'$}
V_\phi ^* :   \maclS _s(\rr {2d}) \to \maclS _s(\rr d),
\qquad
V_\phi ^* :  \maclS _s'(\rr {2d}) \to \maclS _s'(\rr d)
\end{equation}
are continuous.
For $P_\phi$ in \eqref{Eq:ProjphiDef} we have that
\eqref{Eq:ProjphiRule} still holds true and that
\eqref{Eq:STFTProjMaps} can be completed with
\begin{equation}\label{Eq:STFTProjMaps2}
P_\phi (\maclS _s(\rr {2d})) = V_\phi (\maclS _s(\rr d))
\quad \text{and}\quad
P_\phi (\maclS _s'(\rr {2d})) = V_\phi (\maclS _s'(\rr d)).
\end{equation}

\par

We also have that the twisted convolution in \eqref{Eq:TwistConvDef}
is continuous from $\maclS _s(\rr {2d})\times \maclS _s(\rr {2d})$
to $\maclS _s(\rr {2d})$ and uniquely extendable to a continuous
map $\maclS _s(\rr {2d})\times \maclS _s'(\rr {2d})$ or
$\maclS _s'(\rr {2d})\times \maclS _s(\rr {2d})$ to
$\maclS _s'(\rr {2d})$, and that the formulae
\eqref{Eq:ProjOpTwistedConv}--\eqref{Eq:STFTWindTrans}
still hold true
after each $\mascS$ is replaced by $\maclS _s$ in the attached
assumptions.

\par

If instead $s>\frac 12$, then similar facts hold true with $\Sigma _s$
in place of $\maclS _s$ above, at each occurrence.
\end{rem}

\par

\begin{rem}\label{Rem:STFTCharGSFuncDist}
In similar ways as characterizing Gelfand-Shilov spaces
in terms of Fourier estimates (see \eqref{Eq:GSFtransfChar}), we may also
use the short-time Fourier transform to perform similar characterizations.
Moreover, the short-time Fourier transform can in addition be used
to characterize spaces of Gelfand-Shilov distributions.

\par

In fact, let $\phi \in \maclS _s(\rr d)\setminus 0$
($\phi \in \Sigma _s(\rr d)\setminus 0$) be fixed and let $f$
be a Gelfand-Shilov distribution on $\rr d$. Then the following is true:
\begin{enumerate}
\item $f\in \maclS _s(\rr d)$ ($f\in \Sigma _s(\rr d)$), if and only if
\begin{equation}\label{Eq:STFTCharGSFunc}
|V_\phi f(x,\xi )|\lesssim e^{-r(|x|^{\frac 1s} +|\xi |^{\frac 1s})}
\end{equation}
for some $r>0$ (for every $r>0$);

\vrum

\item $f\in \maclS _s'(\rr d)$ ($f\in \Sigma _s'(\rr d)$), if and only if
\begin{equation}\label{Eq:STFTCharGSDist}
|V_\phi f(x,\xi )|\lesssim e^{r(|x|^{\frac 1s} +|\xi |^{\frac 1s})}
\end{equation}
for every $r>0$ (for some $r>0$).
\end{enumerate}
We refer to \cite[Theorem 2.7]{GroZim} for the
characterization (1) concerning Gelfand-Shilov functions
and to \cite[Proposition 2.2]{Toft18}) for the characterization (2)
concerning Gelfand-Shilov distributions.
%
\end{rem}

\par

\subsection{Weight functions}\label{subsec1.4}

\par

A \emph{weight} or \emph{weight function} on $\rr d$ is a
positive function $\omega
\in  L^\infty _{loc}(\rr d)$ such that $1/\omega \in  L^\infty _{loc}(\rr d)$.
The weight $\omega$ is called \emph{moderate},
if there is a positive weight $v$ on $\rr d$ and a constant $C\ge 1$ such that
\begin{equation}\label{moderate}
\omega (x+y) \le C\omega (x)v(y),\qquad x,y\in \rr d.
\end{equation}
If $\omega$ and $v$ are weights on $\rr d$ such that
\eqref{moderate} holds, then $\omega$ is also called
\emph{$v$-moderate}.
We note that \eqref{moderate}
implies that $\omega$ fulfills
the estimates
\begin{equation}\label{moderateconseq}
C^{-1}v(-x)^{-1}\le \omega (x)\le C v(x),\quad x\in \rr d.
\end{equation}
We let $\mascP _E(\rr d)$ be the set of all moderate weights on $\rr d$.

\par

We say that $v$ is \emph{submultiplicative} if
\begin{equation}\label{Eq:Submultiplicative}
v(x+y) \le v(x)v(y)
\quad \text{and}\quad
v(-x)=v(x),\qquad x,y\in \rr d.
\end{equation}
We observe that if
$v\in \mascP _E(\rr d)$ is even and satisfies
\begin{equation}\label{Eq:WeakSubmult}
v(x+y)\le Cv(x)v(y), \qquad x,y\in \rr d,
\end{equation}
for some constant $C>0$, then for
$v_0=C^{1/2}v$, one has that
$v_0\in \mascP _E(\rr d)$ is submultiplicative and
$v\asymp v_0$ (see e.{\,}g. \cite{Fei3,FeiGro1,Gro2}).

\par

We also recall from \cite{Gro2.5}
that if $v$ is positive and locally bounded and
satisfies \eqref{Eq:WeakSubmult},
then $v(x)\le C_0e^{r_0|x|}$ for some positive constants $C_0$ and $r_0$.
In fact, if $x\in \rr d$,
$$
r=\sup _{|x|\le 1}\log v(x),\quad c=\log C
$$
and $n$ is an integer such that $n-1\le |x|\le n$, then
\eqref{Eq:WeakSubmult} gives
$$
v(x) = v(n\cdot (x/n))\le C^nv(x/n)^n \le C^ne^{rn} = e^{(r+c)n}
\le e^{(r+c)(|x|+1)},
$$
which gives the statement.

\par

Therefore, if $v$ is  a submultiplicative weight, then
\begin{equation}\label{Eq:ForSubmult}
v(x)\lesssim e^{r|x|},\qquad x\in \rr d,
\end{equation}
for some $r\ge 0$. Hence, if $\omega \in \mascP _E(\rr d)$, then
\eqref{moderate} and \eqref{Eq:ForSubmult} imply
\begin{equation}\label{Eq:weight0}
\omega (x+y) \lesssim \omega (x) e^{r|y|},\quad x,y\in \rr d
\end{equation}
for some $r>0$. In particular, \eqref{moderateconseq} shows that
for any $\omega _0\in \mascP_E(\rr d)$, there is a constant $r>0$ such that
$$
e^{-r|x|}\lesssim \omega _0(x)\lesssim e^{r|x|},\quad x\in \rr d.
$$

\par

If \eqref{moderate} holds, then there is a smallest positive even function
$v_0$ such that \eqref{moderate}
holds with $C=1$. We remark that this $v_0$ is given by
$$
v_0(x) = \sup _{y\in \rr d} \left ( \frac {\omega (x+y)}{\omega (y)},
\frac {\omega (-x+y)}{\omega (y)}  \right ) ,
$$
and is submultiplicative
(see e.{\,}g. \cite{FeiGro1,Gro1,Toft10}). Consequently, if $\omega$
is a moderate weight, then it is also moderated by submultiplicative
weights. In the sequel, $v$ and $v_j$ for $j\ge 0$, always stand for
submultiplicative weights if
nothing else is stated.

\par

We also remark that in the literature it is common to
define submultiplicative weights as \eqref{Eq:Submultiplicative} should hold,
without the condition $v(-x)=v(x)$, i.{\,}e. that $v$ does not have to
be even (cf. e.{\,}g. \cite{Fei3,FeiGro1,GaSa,Gro2}). However, in the sequel
it is convenient for us to include this property in the definition.

\medspace

There are several subclasses of $\mascP _E(\rr d)$ which are interesting
for different reasons. Though our results later on are formulated
in background of $\mascP _E(\rr d)$, we here mention some subclasses
which especially appear in time-frequency analysis.
First we observe the class $\mascP ^{0} _E(\rr d)$, which consists of all
$\omega\in \mascP _E(\rr d)$
such that \eqref{Eq:weight0} holds for every $r>0$.

\par

The class $\mascP _E^0(\rr d)$ is important
when dealing with spectral invariance for
matrix or convolution operators on
$\ell ^2(\zz d)$ (see e.{\,}g. \cite{Gro3}).
If $v\in \mascP _E(\rr d)$ is submultiplicative,
then $v\in \mascP _E^0(\rr d)$, if and only if
\begin{equation}\label{Eq:GRSCond}
\lim _{n\to \infty} v(nx)^{\frac 1n} = 1
\end{equation}
(see e.{\,}g. \cite{FeGaTo}). The
condition \eqref{Eq:GRSCond} is equivalent to
\begin{equation}\tag*{(\ref{Eq:GRSCond})$'$}
\lim _{n\to \infty} \frac {\log (v(nx))}n = 0,
\end{equation}
and is usually called the \emph{GRS condition}, or
\emph{Gelfand-Raikov-Shilov condition}.

\par

A more restrictive condition on $v$ compared to
\eqref{Eq:GRSCond}$'$ is given
by the Beurling-Domar condition
\begin{equation}\label{Eq:BeurlingDomar}
\sum _{n=1}^\infty \frac {\log (v(nx))}{n^2} <\infty .
\end{equation}
This condition is strongly linked to non quasi-analytic
classes which contain non-trivial compactly supported
elements (see e.{\,}g. \cite{Gro2.5}). Any subexponential
submultiplicative weight satisfies the Beurling-Domar condition. That is,
suppose that $\theta \in (0,1)$ and that $v(x)=e^{r|x|^\theta}$,
$x\in \rr d$, then \eqref{Eq:BeurlingDomar} is fulfilled. We
let $\mascP _{\BD}(\rr d)$ be the set of all weights which
are moderated by submultiplicative weights which satisfy
the Beurling-Domar condition.

\par

Finally we let $\mascP (\rr d)$ be the set of all weights on $\rr d$
which are moderated by polynomially bounded functions. That is,
$\omega \in \mascP (\rr d)$, if and only if there are positive
constants $r$ and $C$ such that
$$
\omega (x+y) \le C\omega (x)(1+|y|)^r,\qquad x,y\in \rr d.
$$
Here we observe that $v(x)=(1+|x|)^r$ is submultiplicative.

\par

Among these weight classes we have
\begin{equation}\label{Eq:OrdringWeightClasses}
\mascP (\rr d)\subsetneq \mascP _{\BD}(\rr d)
\subsetneq
\mascP _E^0(\rr d) \subsetneq \mascP _E(\rr d).
\end{equation}
In fact, it is clear that the ordering in \eqref{Eq:OrdringWeightClasses}
holds. On the other hand, if $r>0$ and $\theta \in (0,1)$, then
due to the submultiplicative weights
\begin{equation}\label{Eq:SomeSubmultWeights}
\begin{aligned}
e^{r|x|^\theta } &\in \mascP _{\BD}(\rr d)\setminus \mascP(\rr d),
\\[1ex]
e^{r|x|/\log (e+|x|)}
&\in
\mascP _E^0(\rr d)\setminus \mascP _{\BD}(\rr d),
\\[1ex]
\text{and}\quad
e^{r|x|} &\in \mascP _E(\rr d)\setminus \mascP _E^0(\rr d),
\end{aligned}
\end{equation}
it also follows that the inclusions in \eqref{Eq:OrdringWeightClasses}
are strict.

\par

We refer to \cite{Fei0,Gro2,Gro2.5,Toft10}
for more facts about weights in time-frequency
analysis.

\par

\subsection{Mixed norm spaces of Lebesgue type}\label{subsec1.5}

\par

For every $p,q\in (0,\infty ]$ and weight $\omega$ on $\rr {2d}$, we set
\begin{alignat*}{3}
\nm {F}{L_{(\omega )}^{p,q}(\rr {2d})}
&\equiv
\nm {G_{F,\omega ,p}}{L^q(\rr d)}, &
\quad &\text{where} & \quad
G_{F,\omega ,p}(\xi ) &= \nm {F(\cdo ,\xi )\omega (\cdo ,\xi )}{L^p(\rr d)}
\intertext{and}
\nm {F}{L^{p,q}_{*,(\omega )}(\rr {2d})}
&\equiv
\nm {H_{F,\omega ,q}}{L^p(\rr d)}, &
\quad &\text{where} & \quad
H_{F,\omega ,q}(x) &= \nm {F(x,\cdo )\omega (x,\cdo )}{L^q(\rr d)},
\end{alignat*}
when $F$ is (complex-valued) measurable function on $\rr {2d}$. Then
$L^{p,q}_{(\omega )}(\rr {2d})$ ($L^{p,q}_{*,(\omega )}(\rr {2d})$)
consists of all measurable functions
$F$ such that $\nm F{L_{(\omega )}^{p,q}}<\infty$
($\nm F{L^{p,q}_{*,(\omega )}}<\infty$).

\par

In similar ways, let $\Omega _1,\Omega _2$ be discrete sets,
$\omega $ be a positive function on $\Omega _1\times \Omega _2$ and
$\ell _0'(\Omega _1\times \Omega _2)$ be the set of all formal (complex-valued)
sequences $c=\{ c(j,k)\} _{j\in \Omega _1,k\in \Omega _2}$. Then
the discrete Lebesgue spaces, i.e. the Lebesgue sequence spaces
$$
\ell ^{p,q}_{(\omega )}(\Omega _1\times \Omega _2)
\quad \text{and}\quad
\ell ^{p,q}_{*,(\omega )}(\Omega _1\times \Omega _2)
$$
of mixed (quasi-)norm types consist of all
$c\in \ell _0'(\Omega _1\times \Omega _2)$
such that
$\nm {c}{\ell _{(\omega )}^{p,q}(\Omega _1\times \Omega _2)}<\infty$
respectively
$\nm {c}{\ell _{*,(\omega )}^{p,q}(\Omega _1\times \Omega _2)}<\infty$.
Here
\begin{alignat*}{3}
\nm {c}{\ell _{(\omega )}^{p,q}(\Omega _1\times \Omega _2)}
&\equiv \nm {G_{F,\omega ,p}}{\ell ^q(\Omega _2)}, &
\quad &\text{where} & \quad
G_{c,p}(k) &= \nm {F(\cdo ,k )\omega (\cdo ,k)}{\ell ^p(\Omega _1)}
\intertext{and}
\nm {c}{\ell ^{p,q}_{*,(\omega )}(\Omega _1\times \Omega _2)}
&\equiv \nm {H_{c,\omega ,q}}{\ell ^p(\Omega _1)}, &
\quad &\text{where} & \quad
H_{c,q}(j) &= \nm {c(j,\cdo )\omega (j,\cdo )}{\ell ^q(\Omega _2)},
\end{alignat*}
when $c\in \ell _0'(\Omega _1\times \Omega _2)$.

\par

\par

\par

\subsection{Convolutions and multiplications for
discrete Lebesgue spaces}\label{subsec1.8}

\par

Next we discuss extended H{\"o}lder and Young relations
for multiplications and convolutions on discrete Lebesgue spaces.
The H{\"o}lder and Young conditions on Lebesgue exponent are then
\begin{align}
\frac 1{q_0} &\le \frac 1{q_1}+\frac 1{q_2},
\label{Eq:LebExpHolder}
\intertext{respectively}
\frac 1{p_0} &\le \frac 1{p_1}+\frac 1{p_2}
-
\max \left ( 1,\frac 1{p_1},\frac 1{p_2}\right ).
\label{Eq:LebExpYoung}
\end{align}

Notice that, when $ p_1, p_2 \in (0,1) $, then \eqref{Eq:LebExpYoung}
becomes $ p_0 \geq \max \{ p_1, p_2 \} $, while for $ p_1, p_2 \geq 1 $
it reduces to the common Young condition
$$
1 + \frac 1{p_0} \le \frac 1{p_1}+\frac 1{p_2}.
$$

The conditions on the weight functions are
\begin{alignat}{2}
\omega _0(j) &\le \omega _1(j)\omega _2(j), &
\qquad
j &\in \Lambda ,
\label{Eq:WeightCondHolder}
\intertext{respectively}
\omega _0(j_1+j_2) &\le \omega _1(j_1)\omega _2(j_2), &
\qquad
j_1,j_2 &\in \Lambda ,
\label{Eq:WeightCondYoung}
\end{alignat}
where $\Lambda$ is a lattice of the form
$$
\Lambda = \sets {n_1e_1+\cdots +n_de_d}
{(n_1,\dots ,n_d)\in \zz d},
$$
where $e_1,\dots e_d$ is a basis for $\rr d$.

\par

\begin{prop}\label{Prop:HolderYoungDiscrLebSpaces}
Let $p_j, q_j\in (0,\infty ]$, $j=0,1,2$, be such that
\eqref{Eq:LebExpHolder} and
\eqref{Eq:LebExpYoung} hold, let
$\Lambda \subseteq \rr d$ be a lattice
and let $\omega _j$ be weights on $\Lambda$, $j=0,1,2$.
Then the following is true:
\begin{enumerate}
\item if \eqref{Eq:WeightCondHolder} holds,
then the map $(a_1,a_2)\mapsto a_1\cdot a_2$ from
$\ell _0(\Lambda ) \times \ell _0(\Lambda )$ to $\ell _0(\Lambda )$
extends uniquely to a continuous map from
$\ell ^{q_1}_{(\omega _1)}(\Lambda )
\times
\ell ^{q_2}_{(\omega _2)}(\Lambda )$
to $\ell ^{q_0}_{(\omega _0)}(\Lambda )$,
and
\begin{equation}\label{Eq:HolderEst}
\nm {a_1\cdot a_2}{\ell ^{q_0}_{(\omega _0)}}
\le
\nm {a_1}{\ell ^{q_1}_{(\omega _1)}}
\nm {a_2}{\ell ^{q_2}_{(\omega _2)}},
\qquad a_j\in \ell ^{q_j}_{(\omega _j)}(\Lambda ),\ j=1,2
\text ;
\end{equation}

\vrum

\item  if \eqref{Eq:WeightCondYoung} holds,
then the map $(a_1,a_2)\mapsto a_1*a_2$ from
$\ell _0(\Lambda )\times \ell _0(\Lambda )$ to $\ell _0(\Lambda )$
extends uniquely to a continuous map from
$\ell ^{p_1}_{(\omega _1)}(\Lambda )
\times \ell ^{p_2}_{(\omega _2)}(\Lambda )$ to
$\ell ^{p_0}_{(\omega _0)}(\Lambda )$,
and
\begin{equation}\label{Eq:YoungEst}
\nm {a_1*a_2}{\ell ^{p_0}_{(\omega _0)}}
\le
\nm {a_1}{\ell ^{p_1}_{(\omega _1)}}\nm {a_2}{\ell ^{p_2}_{(\omega _2)}},
\qquad a_j\in \ell ^{p_j}_{(\omega _j)}(\Lambda ),\ j=1 ,2
\text .
\end{equation}
\end{enumerate}
\end{prop}

\par

The assertion (1) in Proposition \ref{Prop:HolderYoungDiscrLebSpaces}
is the standard
H{\"o}lder's inequality for discrete Lebesgue spaces. The assertion (2) in that
proposition is the usual Young's inequality for Lebesgue spaces on lattices
in the case when $p_0,p_1,p_2\in [1,\infty ]$. A proof of a weighted version
of Proposition \ref{Prop:HolderYoungDiscrLebSpaces} is given in
Appendix A in \cite{Toft26}.

\par

\section{Modulation spaces, multiplications and convolutions}\label{sec2}

\par

In this section we introduce modulation spaces,
and recall their basic properties, in particular in the context of Gelfand-Shilov spaces.
Notice that we permit the Lebesgue exponents to belong to the full interval $(0,\infty ]$
instead of the  most common choice $[1,\infty ]$, and general moderate weights which may have a (sub)exponential growth.
Here we also recall some facts on Gabor expansions for modulation spaces.

Then we deduce multiplication and convolution
estimates on modulation spaces.
There are several approaches to  multiplication and convolution
in the case when the involved Lebesgue exponents belong to
$[1,\infty ]$ (see \cite{CorGro,Fei1,FeiGro1,GuChFaZh,RSTT,Toft3}). Here we consider the case
when these exponents belong $(0,\infty )$ (see also \cite{BaCoNi,BaTe,GaSa,Rau1,Rau2,Toft13}).
In addition, and in order to keep the survey style of our exposition,
we focus on the bilinear case, and refer to \cite{Toft26}
for extension of these results to multi-linear products as well as allowing the
Lebesgue exponents to belong to the full interval $(0,\infty ]$.

\par

\subsection{Modulation spaces}\label{subsec1.6}

\par

The (classical) modulation spaces, essentially introduced in
\cite{Fei1} by Feichtinger are given in the following.
(See e.{\,}g. \cite{Fei6} for definition of more general modulation spaces.)

\par

%
%

\par

\begin{defn}
Let $p,q\in (0,\infty ]$, $\omega \in \mascP _E(\rr {2d})$ and
$\phi \in \Sigma _1 (\rr d)\setminus 0$.
\begin{enumerate}
\item The \emph{modulation space} $M^{p,q}_{(\omega )}(\rr d)$
consists of all $f\in \Sigma _1 '(\rr d)$ such that
$$
\nm f{M^{p,q}_{(\omega )}}\equiv \nm {V_\phi f}{L^{p,q}_{(\omega )}}
$$
is finite. The topology of $M^{p,q}_{(\omega )}(\rr d)$ is defined by
the (quasi-)norm $\nm \cdo{M^{p,q}_{(\omega )}}$;

\vrum

\item The \emph{modulation space (of Wiener amalgam type)}
$W^{p,q}_{(\omega )}(\rr d)$ consists of all $f\in \Sigma _1 '(\rr d)$ such that
$$
\nm f{W^{p,q}_{(\omega )}}\equiv \nm {V_\phi f}{L^{p,q}_{*,(\omega )}}
$$
is finite. The topology of $W^{p,q}_{(\omega )}(\rr d)$ is defined by
the (quasi-)norm $\nm \cdo{W^{p,q}_{(\omega )}}$.
\end{enumerate}
\end{defn}

\par

\begin{rem}\label{Rem:ModSpaces}
Modulation spaces possess several convenient properties.
In fact, let $p,q\in (0,\infty ]$, $\omega \in \mascP _E(\rr {2d})$ and
$\phi \in \Sigma _1(\rr d)\setminus 0$. Then the following is true
(see \cite{Fei1,Fei6,FeiGro1,FeiGro2,GaSa,Gro2} and
their analyses for verifications):
\begin{itemize}
\item the definitions of $M_{(\omega )}^{p,q}(\rr d)$ and
$W_{(\omega )}^{p,q}(\rr d)$
are independent of the choices of $\phi \in \Sigma _1 (\rr d)\setminus 0$,
and different choices give rise to equivalent quasi-norms;

\vrum

\item the spaces $M^{p,q}_{(\omega )}(\rr d)$ and
$W^{p,q}_{(\omega )}(\rr d)$ are quasi-Banach spaces
which increase with $p$ and $q$, and decrease with
$\omega$. If in addition $p,q\ge 1$, then they are Banach spaces;

\vrum

\item If in addition $p,q\ge 1$, then the $L^2(\rr d)$ scalar product,
$(\cdo ,\cdo )_{L^2(\rr d)}$, on $\Sigma _1 (\rr d)\times \Sigma _1 (\rr d)$ is
uniquely extendable to dualities between $M^{p,q}_{(\omega )}(\rr d)$
and $M^{p',q'}_{(1/\omega )}(\rr d)$, and between
$W^{p,q}_{(\omega )}(\rr d)$
and $W^{p',q'}_{(1/\omega )}(\rr d)$. If in addition $p,q<\infty$, then
the dual spaces of $M^{p,q}_{(\omega )}(\rr d)$ and $W^{p,q}_{(\omega )}(\rr d)$
can be identified with $M^{p',q'}_{(1/\omega )}(\rr d)$ respectively
$W^{p',q'}_{(1/\omega )}(\rr d)$, through the form $(\cdo ,\cdo )_{L^2(\rr d)}$;

\vrum

\item if $\omega _0(x,\xi )=\omega (-\xi ,x)$, then
$\mascF$ on $\Sigma _1 '(\rr d)$
restricts to a homeomorphism from $M^{p,q}_{(\omega )}(\rr d)$ to
$W^{q,p}_{(\omega _0)}(\rr d)$.

\vrum

\item The inclusions
\begin{alignat}{4}
\Sigma _1 (\rr d) &\subseteq &
M^{p,q}_{(\omega )}(\rr d),W^{p,q}_{(\omega )}(\rr d)
&\subseteq &\Sigma _1 '(\rr d)
\quad &\text{when} & \quad \omega \in \mascP _E(\rr {2d}),
\label{Eq:ModGSEmbeddings1}
\\[1ex]
\maclS _1 (\rr d) &\subseteq &
M^{p,q}_{(\omega )}(\rr d),W^{p,q}_{(\omega )}(\rr d)
&\subseteq &\maclS _1 '(\rr d)
\quad &\text{when} & \quad \omega \in \mascP _E^0(\rr {2d})
\label{Eq:ModGSEmbeddings2}
\intertext{and}
\mascS (\rr d) &\subseteq &
M^{p,q}_{(\omega )}(\rr d),W^{p,q}_{(\omega )}(\rr d)
&\subseteq & \, \mascS '(\rr d)
\quad &\text{when} & \quad \omega \in \mascP (\rr {2d}).
\label{Eq:ModSchwartzEmbeddings}
\end{alignat}
are continuous. If in addition $p,q<\infty$, then
these inclusions are dense.
\end{itemize}
\end{rem}

\par

We recall from \cite{Toft10} the embeddings
\eqref{Eq:ModGSEmbeddings1}--\eqref{Eq:ModSchwartzEmbeddings},
are essentially special cases of certain
characterizations of the Schwartz space, Gelfand-Shilov
spaces and their distribution spaces in terms of
suitable unions and intersections of modulation spaces. In fact,
let $p,q\in (0,\infty ]$ and $s\ge 1$ be fixed and set
\begin{equation}
v_{r,t}(x,\xi ) =
\begin{cases}
e^{r(|x|^{\frac 1t}+|\xi |^{\frac 1t}))}, & t\in \mathbf R_+
\\[1ex]
(1+|x|+|\xi |)^r, & t=\infty .
\end{cases}
\end{equation}
Then
\begin{alignat}{2}
\Sigma _s(\rr d) &= \bigcap _{r>0}M^{p,q}_{(v_{r,s})}(\rr d) &
&= \bigcap _{r>0}W^{p,q}_{(v_{r,s})}(\rr d),
\label{Eq:SchwartzGSModChar1}
\\[1ex]
\maclS _s(\rr d) &= \bigcup _{r>0}M^{p,q}_{(v_{r,s})}(\rr d) &
&= \bigcup _{r>0}W^{p,q}_{(v_{r,s})}(\rr d),
\label{Eq:SchwartzGSModChar2}
\\[1ex]
\mascS (\rr d) &= \bigcap _{r>0}M^{p,q}_{(v_{r,\infty })}(\rr d) &
&= \bigcap _{r>0}W^{p,q}_{(v_{r,\infty})}(\rr d),
\label{Eq:SchwartzGSModChar3}
\\[1ex]
\mascS '(\rr d) &= \bigcup _{r>0}M^{p,q}_{(1/v_{r,\infty })}(\rr d) &
&= \bigcup _{r>0}W^{p,q}_{(1/v_{r,\infty})}(\rr d),
\label{Eq:SchwartzGSModChar4}
\\[1ex]
\maclS _s'(\rr d) &= \bigcap _{r>0}M^{p,q}_{(1/v_{r,s})}(\rr d) &
&= \bigcap _{r>0}W^{p,q}_{(1/v_{r,s})}(\rr d)
\label{Eq:SchwartzGSModChar5}
\intertext{and}
\Sigma _s'(\rr d) &= \bigcup _{r>0}M^{p,q}_{(1/v_{r,s})}(\rr d) &
&= \bigcup _{r>0}W^{p,q}_{(1/v_{r,s})}(\rr d).
\label{Eq:SchwartzGSModChar6}
\end{alignat}
The topologies of the spaces on the left-hand sides of
\eqref{Eq:SchwartzGSModChar1}--\eqref{Eq:SchwartzGSModChar6}
are obtained by replacing each intersection
by projective limit with respect to $r>0$ and each union
with inductive limit with respect to $r>0$.

\par

The relations \eqref{Eq:SchwartzGSModChar1}--\eqref{Eq:SchwartzGSModChar6}
are essentially special cases of \cite[Theorem 3.9]{Toft10}, see also
\cite{GroZim, Teofanov2, Teofanov3}. In order
to be self-contained we here give a proof of \eqref{Eq:SchwartzGSModChar2}.

\par

\begin{proof}[Proof of \eqref{Eq:SchwartzGSModChar2}]
Since
$$
M^\infty _{(v_{2r,s})}(\rr d) \subseteq
M^{p,q}_{(v_{r,s})}(\rr d),W^{p,q}_{(v_{r,s})}(\rr d)
\subseteq M^\infty _{(v_{r,s})}(\rr d),
$$
it suffices to prove the result for $p=q=\infty$. Let
$\phi \in \Sigma _1(\rr d)\setminus 0$ be fixed. First suppose that
$$
f\in M^\infty _{(v_{r,s})}(\rr d)=W^\infty _{(v_{r,s})}(\rr d).
$$
Then it follows from the definition of modulation space norm that
\eqref{Eq:STFTCharGSFunc} holds for some $r>0$. By Remark
\ref{Rem:STFTCharGSFuncDist} it follows that $f\in \maclS _s(\rr d)$,
and we have proved
\begin{equation}\label{Eq:SchwartzGSModChar2A}
\bigcup _{r>0}M^\infty _{(v_{r,s})}(\rr d)
\subseteq \maclS _s(\rr d).
\end{equation}

\par

Suppose instead that $f\in \maclS _s(\rr d)$. Then
\eqref{Eq:STFTCharGSFunc} holds for some $r>0$, giving that
$f\in M^\infty _{(v_{r,s})}(\rr d)$. Hence \eqref{Eq:SchwartzGSModChar2A}
holds with reversed inclusion, and the result follows.
\end{proof}

\par

\begin{example}\label{ExampleSzero}
Let $ p=q=1$ and $\omega =1$. Then  $M^{1,1} _{(\omega )}
(\rr d) = M^{1} (\rr d)$ is {\em the Feichtinger algebra},
probably the most prominent example of a modulation space. We refer to a recent survey \cite{Jakobsen} for a detailed account on $M^{1} (\rr d)$, and to \cite[Lemma 11]{DPT2022} for a list of its basic properties.

Familiar examples arise when $ p=q=2$.
Then $M_{(\omega )} ^{2,2} (\rr d) =M^{2} (\rr d)  = L^2 (\rr d)$, and
\begin{equation*}
M_{(\omega_s)} ^{2,2} (\rr d) = H^s  (\rr d),  \qquad s \in \mathbf{R},
\label{eqModulationSpace1.1}
\end{equation*}
where $\omega_s(\xi) = \langle \xi \rangle^{s}$, and  $H^s (\rr d)$ is the Sobolev space (also known as the Bessel potential space) of distributions
$f \in \mathcal S^{\prime}(\rr d) $ such that
\begin{equation*}
\|f\|_{H^s}^2 := \int_{\rr d} \langle \xi \rangle^{2s} | \widehat{f} (\xi)|^2 d \xi < \infty,
\label{eqModulationSpace1.2}
\end{equation*}
cf. \cite[Proposition 11.3.1]{Gro2}. Furthermore,
if $v_{s} (x,\xi) =  \langle (x,\xi) \rangle^{s},$ then
$ M_{(v_{s})} ^{2,2}  (\rr d) = Q_s  (\rr d)$, $s \in \mathbf{R},$
 \cite[Lemma 2.3]{Boggiatto}. Here $ Q_s $ denotes the Shubin-Sobolev space, \cite{Shubin1}.
\end{example}

\par

Finally we remark that modulation spaces can be
conveniently discretized in terms of Gabor expansions.
In order for explaining some basic issues on this, in similar ways as
in Subsection 1.5 in \cite{Toft26}, we limit
ourself to the case when the involved weights are moderated
by subexponential functions. That is, we suppose that $\omega$
in $M^{p,q}_{(\omega )}(\rr d)$ satisfies
\begin{equation}\label{Eq:SubexpMod}
\omega (x+y,\xi +\eta )
\lesssim
\omega (x,\xi )e^{r(|x|^{\frac 1s}+|\xi |^{\frac 1s})},
\end{equation}
for some $s>1$ and $r>0$. We observe that this implies that
\begin{equation}
\Sigma _s(\rr d) \subseteq M^{p,q}_{(\omega )}(\rr d)
\subseteq \Sigma _s'(\rr d),
\end{equation}
in vew of \eqref{moderateconseq}, \eqref{Eq:SchwartzGSModChar1}
and \eqref{Eq:SchwartzGSModChar6}. For more general approaches
we refer to \cite{FeiGro1,Gro1,Gro2,Rau2,Toft13}.

\par

Since $s>1$, it follows from Sections 1.3 and 1.4 in \cite{Ho1}
that there are $\phi ,\psi \in \Sigma _s (\rr d)$ with values in
$[0,1]$ such that
\begin{alignat}{4}
\supp \phi &\subseteq \Big [-\frac 34,\frac 34\Big ] ^d,&
\quad
\phi (x) &= 1&
\quad &\text{when}& \quad
x&\in \Big [-\frac 14,\frac 14\Big ]^d
\label{Eq:phiProp}
\\[1ex]
\supp \psi &\subseteq [-1,1]^d,&
\quad
\psi (x) &= 1&
\quad &\text{when}& \quad
x&\in \Big [-\frac 34,\frac 34\Big ]^d
\label{Eq:psiProp}
\end{alignat}
and
\begin{equation}\label{Eq:PartUnity1}
\sum _{j\in \zz d}\phi (\cdo -j) =1.
\end{equation}

\par

Let $f\in \Sigma _s'(\rr d)$. Then
$x\mapsto f(x)\phi (x-j)$ belongs to $\Sigma _s'(\rr d)$
and is supported in $j+[-\frac 34,\frac 34]^d$. Hence, by
periodization it follows from Fourier analysis that
\begin{equation}\label{Eq:FirstExp}
f(x)\phi (x-j) = \sum _{\iota \in \pi \zz d}c(j,\iota )e^{i\scal x\iota},
\qquad
x\in j+[-1,1]^d,
\end{equation}
where
$$
c(j,\iota ) = 2^{-d}(f,\phi (\cdo -j)e^{i\scal \cdo \iota})
=
\left ( \frac \pi 2\right )^{\frac d2} V_\phi f(j,\iota ),
\qquad j\in \zz d,\ \iota \in \pi \zz d .
$$
Since $\psi =1$ on the support of $\phi$, \eqref{Eq:FirstExp} gives
\begin{equation}\tag*{(\ref{Eq:FirstExp})$'$}
f(x)\phi (x-j) = \left ( \frac \pi 2\right )^{\frac d2}
\sum _{\iota \in \pi \zz d}V_\phi f(j,\iota )\psi (x-j)e^{i\scal x\iota},
\qquad
x\in \rr d,
\end{equation}
By \eqref{Eq:PartUnity1} it now follows that
\begin{align}
f(x) &= \left ( \frac \pi 2\right )^{\frac d2}
\sum _{(j,\iota )\in \Lambda}V_\phi f(j,\iota )\psi (x-j)e^{i\scal x\iota},
\qquad
x\in \rr d,
\label{Eq:SpecGaborExp1}
\intertext{where}
\Lambda &= \zz d\times (\pi \zz d),
\label{Eq:OurLattice}
\end{align}
which is the \emph{Gabor expansion} of $f$ with respect to the
\emph{Gabor pair} $(\phi ,\psi )$ and lattice $\Lambda$,
i.{\,}e. with respect to the \emph{Gabor atom}
$\phi$ and the \emph{dual Gabor atom} $\psi$.
Here the series converges in
$\Sigma _s'(\rr d)$. By duality and the fact that
compactly supported elements in
$\Sigma _s(\rr d)$ are dense in $\Sigma _s'(\rr d)$
we also have
\begin{equation}\label{Eq:SpecGaborExp2}
f(x) = \left ( \frac \pi 2\right )^{\frac d2}
\sum _{(j,\iota )\in \Lambda}V_\psi f(j,\iota )\phi (x-j)e^{i\scal x\iota},
\qquad
x\in \rr d,
\end{equation}
with convergence in $\Sigma _s'(\rr d)$.

\par

Let $T$ be a linear continuous operator from $\Sigma _s(\rr d)$ to
$\Sigma _s'(\rr d)$ and let $f\in \Sigma _s(\rr d)$. Then
it follows from \eqref{Eq:SpecGaborExp1} that
$$
(Tf)(x) = \left ( \frac \pi 2\right )^{\frac d2}
\sum _{(j,\iota )\in \Lambda}V_\phi f(j,\iota )T(\psi (\cdo -j)e^{i\scal \cdo \iota})(x)
$$
and
$$
T(\psi (\cdo -j)e^{i\scal \cdo \iota})(x)
=
\left ( \frac \pi 2\right )^{\frac d2}
\sum _{(k,\kappa )\in \Lambda}(V_\phi (T(\psi (\cdo -j)e^{i\scal \cdo \iota})))(k,\kappa )
\psi (x-k)e^{i\scal x\kappa}.
$$
A combination of these expansions show that
\begin{equation}\label{Eq:SpecGaborExp3}
(Tf)(x) = \left ( \frac \pi 2\right )^{\frac d2}
\sum _{(j,\iota )\in \Lambda} (A\cdot V_\phi f)(j,\iota )\psi (x-j)e^{i\scal x\iota},
\end{equation}
where $A=(a(\mabfj ,K))_{\mabfj ,\mabfk \in \Lambda}$ is the
$\Lambda \times \Lambda$-matrix, given by
\begin{multline}\label{Eq:OpGaborMatrix}
a(\mabfj ,\mabfk ) = \left ( \frac \pi 2\right )^{\frac d2}
(T(\psi (\cdo -j)e^{i\scal \cdo \iota}) , \phi (\cdo -k)e^{i\scal \cdo \kappa} )_{L^2(\rr d)}
\\[1ex]
\text{when} \quad \mabfj =(j,\iota ) \ \text{and} \ \mabfk =(k,\kappa ).
\end{multline}

\par

By the Gabor analysis for modulation spaces we get the following
restatement of \cite[Proposition 1.8]{Toft26}. We
refer to \cite{Fei3,FeiGro1,FeiGro2,FeiGro3,GaSa,Gro1,Gro2,Toft13} for details.

\par

\begin{prop}\label{Prop:GaborExpMod}
Let $s> 1$, $p,q\in (0,\infty ]$, $\omega \in \mascP _E
(\rr {2d})$ be such that \eqref{Eq:SubexpMod} holds for some
$r>0$,
$\phi ,\psi \in \Sigma _s (\rr d) $ with values in $[0,1]$ be such that
\eqref{Eq:phiProp}, \eqref{Eq:psiProp} and \eqref{Eq:PartUnity1}
hold true, and let $f\in \Sigma _s'(\rr d)$. Then the following is true:
\begin{enumerate}
\item $f\in M^{p,q}_{(\omega )}(\rr d)$, if and only if
$\nm {V_\phi f}{\ell _{(\omega )}^{p,q}(\zz d\times \pi \zz d)}$;

\vrum

\item $f\in M^{p,q}_{(\omega )}(\rr d)$, if and only if
$\nm {V_\psi f}{\ell _{(\omega )}^{p,q}(\zz d\times \pi \zz d)}$;

\vrum

\item the quasi-norms
$$
f\mapsto \nm {V_\phi f}{\ell _{(\omega )}^{p,q}(\zz d\times \pi \zz d)}
\quad \text{and}\quad
f\mapsto \nm {V_\psi f}{\ell _{(\omega )}^{p,q}(\zz d\times \pi \zz d)}
$$
are equivalent to $\nm \cdo{M^{p,q}_{(\omega )}}$.
\end{enumerate}
The same holds true with $W^{p,q}_{(\omega )}$
and $\ell _{*,(\omega )}^{p,q}$ in place of
$M^{p,q}_{(\omega )}$
respectively $\ell _{(\omega )}^{p,q}$ at each occurrence.
\end{prop}

%
%

\subsection{Multiplications and
convolutions in modulation spaces}\label{subsec2.2}

As a first step for approaching multiplications and
convolutions for elements in modulation spaces, we
reformulate such products in terms of short-time Fourier transforms.
Let $\phi _0,\phi _1 ,\phi _2\in \Sigma _1 (\rr d)$ be fixed
such that
\begin{equation}\label{Eq:MultDefFormula2A}
\phi _0=(2\pi )^{-\frac {d}2}\phi _1\phi _2
\end{equation}
and let $f_1,f_2\in \Sigma _1 (\rr d)$. Then the multiplication
$f_1f_2$ can be expressed by
\begin{align}
F_0(x,\xi )
&=
\big ( F_1(x,\cdo ) * F_2(x,\cdo ) \big ) (\xi ).
\label{Eq:MultDefFormula1Mod}
\intertext{where}
F_j &= V_{\phi _j}f_j,\qquad j=0,1,2,
\label{Eq:MultDefFormula2}
\end{align}

\par

In fact, by Fourier's inversion formula we get
\begin{multline*}
\big ( (V_{\phi _1}f_1)(x,\cdo )
*
(V_{\phi _2}f_2)(x,\cdo ) \big ) (\xi )
\\[1ex]
=
(2\pi )^{-d} \iiint
f_1(y_1)\overline{\phi _1(y_1-x)}
f_2(y_2)\overline{\phi _2(y_2-x)}
e^{-{i\scal {y_1}{\xi -\eta}}} e^{-i\scal {y_2}\eta}
\, dy_1dy_2d\eta
\\[1ex]
=
\int f_1(y)\overline{\phi _1(y-x)}
f_2(y)\overline{\phi _2(y-x)}e^{-i\scal y\xi}\, dy
=(2\pi )^{\frac d2}(V_{\phi _1\phi _2}(f_1f_2))(x,\xi ).
\end{multline*}

\par

We also observe that we may extract $f_0=f_1f_2$ by the formula
\begin{equation}\label{Eq:f0Extract}
f_0 = (\nm {\phi _0}{L^2})^{-1}V_{\phi _0}^*F_0 ,
\end{equation}
provided $\phi _0$ is not trivially equal to $0$.

\par

In the same way, let $\phi _0,\phi _1,\phi _2\in \Sigma _1 (\rr d)$ be
fixed such that
\begin{equation}\label{Eq:ConvDefFormula2A}
\phi _0=(2\pi )^{\frac d2}\phi _1*\phi _2,
\end{equation}
and let $f_1,f_2,g\in \Sigma _1 (\rr d)$. Then the
convolution $f_1*f_2$ can be expressed by
\begin{align}
F_0(x,\xi )
&=
\big ( F_1(\cdo ,\xi ) *F_2(\cdo ,\xi ) \big ) (x).
\label{Eq:ConvDefFormula1Mod}
\end{align}
where $F_j$ are given by \eqref{Eq:MultDefFormula2},
and that we may extract $f_0=f_1*f_2$ from
\eqref{Eq:f0Extract}.

\par

Next we discuss convolutions and multiplications for modulation spaces,
and start with the following convolution result for modulation spaces.
For multiplications of elements in modulation spaces we need
to swap the conditions for the involved Lebesgue exponents
compared to \eqref{Eq:LebExpHolder}
and \eqref{Eq:LebExpYoung}. That is, these conditions become
\begin{alignat}{2}
\frac 1{p_0} &\le  \frac 1{p_1}+\frac 1{p_2},&
\qquad
\frac 1{q_0} &\le  \frac 1{q_1}+\frac 1{q_2}
- \max \left (
1,\frac 1{p_0},\frac 1{q_1},\frac 1{q_2}
\right ),
\label{Eq:LebExpHolderYoung1}
\intertext{or}
\frac 1{p_0} &\le  \frac 1{p_1}+\frac 1{p_2},&
\qquad
\frac 1{q_0} &\le  \frac 1{q_1}+\frac 1{q_2}
- \max \left (
1,\frac 1{q_1},\frac 1{q_2}
\right ).
\label{Eq:LebExpHolderYoung2}
\end{alignat}
The conditions on the weight functions are
\begin{alignat}{2}
\omega _0(x,\xi _1+\xi _2) &\le \omega _1(x,\xi _1)\omega _2(x,\xi _2), &
\qquad
x,\xi _1,\xi _2 &\in \rr d,
\label{Eq:WeightCondHolderMod}
\intertext{respectively}
\omega _0(x_1+x_2,\xi ) &\le \omega _1(x_1,\xi )\omega _2(x_2,\xi ), &
\qquad
x_1,x_2,\xi &\in \rr d .
\label{Eq:WeightCondYoungMod}
\end{alignat}

\par

\begin{thm}\label{Thm:MultMod1}
Let $p_j, q_j\in (0,\infty )$ and $\omega _j\in \mascP _E(\rr {2d})$, $j=0,1,2$,
be such that \eqref{Eq:LebExpHolderYoung1} and
\eqref{Eq:WeightCondHolderMod} hold. Then
$(f_1,f_2)\mapsto f_1f_2$
from $\Sigma _1 (\rr d)\times \Sigma _1(\rr d)$
to $\Sigma _1 (\rr d)$ is uniquely extendable to
a continuous map from $M^{p_1,q_1}_{(\omega _1)}(\rr d)
\times M^{p_2,q_2}_{(\omega _2)}(\rr d)$
to $M_{(\omega _0)}^{p_0,q_0}(\rr d)$, and
\begin{equation}\label{Eq:MultMod1}
\nm {f_1f_2}{M_{(\omega _0)}^{p_0,q_0}}
\lesssim
\nm {f_1}{M_{(\omega _1)}^{p_1,q_1}}\nm {f_2}{M_{(\omega _2)}^{p_2,q_2}} ,
\quad
f_j\in M_{(\omega _j)}^{p_j,q_j}(\rr d),\ j=1,2 \text .
\end{equation}
\end{thm}

\par

\begin{thm}\label{Thm:MultMod2}
Let $p_j, q_j\in (0,\infty )$ and $\omega _j\in \mascP _E(\rr {2d})$, $j=0,1,2$,
be such that \eqref{Eq:LebExpHolderYoung2} and
\eqref{Eq:WeightCondHolderMod} hold. Then
$(f_1,f_2)\mapsto f_1f_2$
from $\Sigma _1 (\rr d)\times \Sigma _1(\rr d)$
to $\Sigma _1 (\rr d)$ is uniquely extendable to
a continuous map from $W^{p_1,q_1}_{(\omega _1)}(\rr d)
\times W_{(\omega _2)}^{p_2,q_2}(\rr d)$
to $W_{(\omega _0)}^{p_0,q_0}(\rr d)$, and
\begin{equation}\label{Eq:MultMod2}
\nm {f_1f_2}{W_{(\omega _0)}^{p_0,q_0}}
\lesssim
\nm {f_1}{W_{(\omega _1)}^{p_1,q_1}}
\nm {f_2}{W_{(\omega _2)}^{p_2,q_2}} ,
\quad
f_j\in W_{(\omega _j)}^{p_j,q_j}(\rr d),\ j=1,2 \text .
\end{equation}
\end{thm}

\par

The corresponding results for convolutions are the following. Here
the conditions on the involved Lebesgue exponents are swapped as
\begin{alignat}{2}
\frac 1{p_0} &\le \frac 1{p_1}+\frac 1{p_2} -
\max \left (
1,\frac 1{q_0},\frac 1{p_1},\frac 1{p_2}
\right ), &
\qquad
\frac 1{q_0} &\le \frac 1{q_1}+\frac 1{q_2}
\label{Eq:LebExpYoungHolder1}
\intertext{or}
\frac 1{p_0} &\le \frac 1{p_1}+\frac 1{p_2} -
\max \left (
1,\frac 1{p_1},\frac 1{p_2}
\right ), &
\qquad
\frac 1{q_0} &\le \frac 1{q_1}+\frac 1{q_2}
\label{Eq:LebExpYoungHolder2}
\end{alignat}

\par

\begin{thm}\label{Thm:ConvMod1}
Let $p_j, q_j\in (0,\infty )$ and $\omega _j\in \mascP _E(\rr {2d})$, $j=0,1,2$,
be such that \eqref{Eq:WeightCondYoungMod}
and \eqref{Eq:LebExpYoungHolder2} hold. Then
$(f_1,f_2)\mapsto f_1*f_2$
from $\Sigma _1 (\rr d)\times \Sigma _1(\rr d)$
to $\Sigma _1 (\rr d)$ is uniquely extendable to
a continuous map from $M^{p_1,q_1}_{(\omega _1)}(\rr d)
\times M^{p_2,q_2}_{(\omega _2)}(\rr d)$
to $M_{(\omega _0)}^{p_0,q_0}(\rr d)$, and
\begin{equation}\label{Eq:ConvMod1}
\nm {f_1*f_2}{M_{(\omega _0)}^{p_0,q_0}}
\lesssim
\nm {f_1}{M_{(\omega _1)}^{p_1,q_1}}\nm {f_2}{M_{(\omega _2)}^{p_2,q_2}} ,
\quad
f_j\in M_{(\omega _j)}^{p_j,q_j}(\rr d),\ j=1,2 \text .
\end{equation}
\end{thm}

\par

\begin{thm}\label{Thm:ConvMod2}
Let $p_j, q_j\in (0,\infty )$ and $\omega _j\in \mascP _E(\rr {2d})$, $j=0,1,2$,
be such that \eqref{Eq:WeightCondYoungMod}
and \eqref{Eq:LebExpYoungHolder1} hold. Then
$(f_1,f_2)\mapsto f_1*f_2$
from $\Sigma _1 (\rr d)\times \Sigma _1(\rr d)$
to $\Sigma _1 (\rr d)$ is uniquely extendable to
a continuous map from $W^{p_1,q_1}_{(\omega _1)}(\rr d)
\times W^{p_2,q_2}_{(\omega _2)}(\rr d)$
to $W_{(\omega _0)}^{p_0,q_0}(\rr d)$, and
\begin{equation}\label{Eq:ConvMod2}
\nm {f_1*f_2}{W_{(\omega _0)}^{p_0,q_0}}
\lesssim
\nm {f_1}{W_{(\omega _1)}^{p_1,q_1}}\nm {f_2}{W_{(\omega _2)}^{p_2,q_2}} ,
\quad
f_j\in W_{(\omega _j)}^{p_j,q_j}(\rr d),\ j=1,2 \text .
\end{equation}
\end{thm}

\par

We observe that Theorems 3.2--3.5 in \cite{Toft26} are
multi-linear
versions of the previous results. In particular,
Theorems \ref{Thm:MultMod1}--\ref{Thm:ConvMod2}
(which are special cases of Theorems 3.2--3.5 in \cite{Toft26})
are Fourier transformations of Theorems \ref{Thm:ConvMod1} and \ref{Thm:ConvMod2}.
Hence it suffices to prove the last two theorems, cf.  \cite{Toft26}.
To shed some ideas of the arguments, we give a proof
in the unweighted case of Theorem \ref{Thm:ConvMod1}.
We will use Proposition \ref{Prop:ProjMapModCont}
from Appendix \ref{App:A}, which is a
special case of \cite[Proposition 3.6]{Toft26}.

\par



\par

\begin{proof}[Proof of Theorem \ref{Thm:ConvMod1}]
%
%
Suppose
$f_j\in \mascS (\rr d)$,
$\phi _j\in \mascS (\rr d)$, $j=0,1,2$ be such that
$$
f_0=f_1*f_2
\quad \text{and}\quad
\phi _0=(2\pi) ^{\frac d2}\phi _1*\phi _2\neq 0,
$$
and let $F_j$ be the same as in \eqref{Eq:MultDefFormula2}.
Then
$$
F_0(x,\xi ) = (V_{\phi _1}f_1(\cdo ,\xi )*V_{\phi _2}f_2(\cdo ,\xi ))(x),
$$
in view of \eqref{Eq:ConvDefFormula1Mod}.

\par

We have
$$
0\le \chi _{k_1+Q}*\chi _{k_2+Q} \le \chi _{k_1+k_2+Q_{d,2}},
\qquad
k_1,k_2\in \zz d,
$$
where $Q_{d,r}$ is the cube
$$
Q_{d,r} =[0,r]^d
\quad \text{and}\quad
Q= Q_{d,1}= [0,1]^d,
$$
and $\chi _E$ is the characteristic function with respect to
the set $E$.

\par

Set
\begin{align*}
G(x,\xi ) &= (|V_{\phi _1}f_1(\cdo ,\xi )|*|V_{\phi _2}f_2(\cdo ,\xi )|)(x),
\\[1ex]
a_j(k,\kappa ) &= \nm {V_{\phi _j}f_j}{L^\infty ((k,\kappa )+Q_{2d,1})}, \quad j=1,2,
\intertext{and}
b(k,\kappa ) &=  \nm {G}{L^\infty ((k,\kappa )+Q_{2d,1})}
\end{align*}
Then
\begin{alignat}{2}
\nm {V_{\phi _0}^*F_0}{M^{p_0,q_0}}
&\asymp
\nm {P_{\phi _0}F_0}{\sfW (\ell ^{p_0,q_0})}
\lesssim
\nm {F_0}{\sfW (\ell ^{p_0,q_0})}
\notag
\\[1ex]
&\le
\nm {G}{\sfW (\ell ^{p_0,q_0})}
\asymp
\nm b{\ell ^{p_0,q_0}},
\label{Eq:ModCoorbitNormRel}
\intertext{and}
\nm {f_j}{M^{p_j,q_j}} &\asymp \nm {a_j}{\ell ^{p_j,q_j}}
\label{Eq:ModCoorbitNormRel2}
\end{alignat}
in view of \eqref{Eq:ModWienerConn1} and
Proposition \ref{Prop:ProjMapModCont} in Appendix
\ref{App:A}
(see also \cite[Theorem 3.3]{GaSa})).

\par

By \eqref{Eq:ConvDefFormula1Mod} we have
\begin{equation}\label{Eq:G(x,lambda)Est}
\begin{aligned}
G(x,\lambda )
&\le
\sum _{k_1,k_2\in \zz d} a_1(k_1,\lambda )a_2(k_2,\lambda )
(\chi _{k_1+Q}*\chi _{k_2+Q})(x)
\\[1ex]
&\le
\sum _{k_1,k_2\in \zz d} a_1(k_1,\lambda )a_2(k_2,\lambda )
\chi _{k_1+k_2+Q_{d,2}}(x).
\end{aligned}
\end{equation}
We observe that
$$
\chi _{k_1+k_2+Q_{d,2}}(x)=0
\quad \text{when}\quad
x\notin l+Q_d,\ (k_1,k_2)\notin \Omega _l,
$$
where
$$
\Omega _l
=
\sets {(k_1,k_2)\in \zz {2d}}{l_j-2 \le k_{1,j} +k_{2,j}\le l_j+1},
$$
and
$$
k_j=(k_{j,1},\dots ,k_{j,d})\in \zz d,
\qquad j=1,2,
\quad \text{and}\quad
l=(l_1,\dots ,l_d)\in \zz d.
$$
Hence, if $x=l$ in \eqref{Eq:G(x,lambda)Est}, we get
\begin{multline}\label{Eq:GExprEst}
b(l,\lambda )
\le
\sum _{(k_1,k_2)\in \Omega _l} a_1(k_1,\lambda )a_2(k_2,\lambda )
\\[1ex]
\le
\sum _{m\in I}(a_1(\cdo ,\lambda )*a_2(\cdo ,\lambda ))(l-2e_0+m),
\end{multline}
where $e_0=(1,\dots ,1)\in \zz d$ and $I = \{ 0,1,2,3\} ^d$.

\par

If we apply the $\ell ^{p_0}$ quasi-norm on \eqref{Eq:GExprEst}
with respect to the $l$ variable, then Proposition
\ref{Prop:HolderYoungDiscrLebSpaces} (2) and the fact that
$I$ is finite set give
\begin{multline*}
\nm {b(\cdo ,\lambda )}{\ell ^{p_0}}
\le
\NM {\sum _{m\in I} ( a_{1}(\cdo ,\lambda )
*a_{2}(\cdo ,\lambda ))
(\cdo -2e_0+m)}{\ell ^{p_0}}
\\[1ex]
\le
\sum _{m\in I} \nm {( a_{1}(\cdo ,\lambda )
*a_{2}(\cdo ,\lambda ))
(\cdo -2e_0+m)}{\ell ^{p_0}}
\\[1ex]
\asymp
\nm {a_{1}(\cdo ,\lambda )
*a_{2}(\cdo ,\lambda )}{\ell ^{p_0}}
\\[1ex]
\le
\nm {a_{1}(\cdo ,\lambda )}{\ell ^{p_1}}
\nm {a_{2}(\cdo ,\lambda )}{\ell ^{p_2}}.
\end{multline*}
By applying the $\ell ^{q_0}$ quasi-norm and using
Proposition
\ref{Prop:HolderYoungDiscrLebSpaces} (1)
we now get
$$
\nm {b}{\ell ^{p_0,q_0}}
\lesssim
\nm {a_{1}}{\ell ^{p_1,q_1}}\nm {a_{2}}{\ell ^{p_2,q_2}}.
$$
This is the same as
$$
\nm {G}{L^{p_0,q_0}}
\lesssim
\nm {F_1}{L^{p_1,q_1}}\nm {F_2}{L^{p_2,q_2}}.
$$
A combination of this estimate with \eqref{Eq:ModCoorbitNormRel} and
\eqref{Eq:ModCoorbitNormRel2} gives that $f_1*f_2$ is well-defined
and that \eqref{Eq:ConvMod1} holds.

\par

The uniqueness now follows from that \eqref{Eq:ConvMod1}
holds for $f_1,f_2\in \mascS (\rr d)$, and that
$\mascS (\rr d)$ is dense in $M^{p,q}(\rr d)$ when
$p,q<\infty$.
\end{proof}

\par

\section{Gabor products and modulation spaces}\label{Sec:3}

\par

In this section we give an illustration how the multiplication
properties for modulation spaces can be used when treating
certain nonlinear problems. We consider the Gabor product
which is connected to such multiplication properties.
It is introduced in \cite{DPT2022} in order
to derive a phase space analogue to the usual
convolution identity for the Fourier transform
\eqref{Eq:FourTransfConv}.
We will prove a formula related to
\eqref{Eq:MultDefFormula1Mod}, and then use results
from previous section to extend the
Gabor product initially defined on
$ M^1 (\mathbf{R}^{2d}) \times  M^1 (\mathbf{R}^{2d}) $
to some other spaces. Finally, we show how the Gabor product
gives rise to a phase-space formulation of the qubic
Schr{\"o}dinger equation.



\par

\begin{defn}\label{definitionGaborProduct}
Let $\phi \in M^1 (\mathbf{R}^{d}) \setminus 0$,
and let $F_1 , F_2 \in M^1 (\rr {2d})$.
Then the  Gabor product $\natural_{\phi}$
is given by
\begin{multline}
\label{eqWindowedProduct1}
\left(F_1 \natural_{\phi}\,  F_2 \right) (x, \xi)
\\[1ex]
=
(2\pi )^{-d}
e^{-i\scal x \xi } \iiint_{\rr {3d} }
\overline{\widehat{\phi}
\left(\zeta -\xi \right)}
e^{i\scal x\zeta} F_1 (y,\eta ) F_2
(y, \zeta - \eta )\,  dyd\eta d\zeta .
\end{multline}
\end{defn}

\par

In the proof of \cite[Lemma 13]{DPT2022} it is justified that
the Gabor product  in  \eqref{eqWindowedProduct1} is well-defined,
and  that
$$
\natural _\phi  : M^1 (\mathbf{R}^{2d})\times M^1
(\mathbf{R}^{2d}) \to M^1 (\mathbf{R}^{2d})
$$
is a continuous map.

\par

The Gabor product is particularly well-suited in the
context of the STFT.

\begin{thm}
Let $\phi , \phi _1, \phi _2
\in
M^1 (\mathbf{R}^{d}) \backslash \left\{0 \right\}$.
Then
\begin{equation}
( \phi _2,\phi _1 )_{L^2 (\rr d)}
V_{\phi} (f_1 \cdot f_2) =   (V_{\phi _1}
f_1) \natural_{\phi} (V_{\overline \phi _2} f_2),
\quad f_1, f_2 \in  M^1 (\mathbf{R}^{d}).
\label{eqWindowedProduct2}
\end{equation}
Moreover, $ V_{g} (f_1 \cdot f_2)  \in M^1 (\mathbf{R}^{2d})$.
\end{thm}


\par

\begin{proof}
We have
\begin{align}
((V_{\phi _1}f_1) &\natural _\phi (V_{\overline \phi _2}f_2))(x,\xi )
\\[1ex]
&=
(2\pi )^{-d}e^{-i\scal x\xi}
\iint _{\rr {2d}}
\overline {\widehat \phi (\zeta -\xi )}e^{i\scal x\zeta}G(y,\zeta )\, dyd\zeta ,
\label{Eq:ReformNaturalProd}
\intertext{where}
G(y,\zeta ) &= \int _{\rr d} (V_{\phi _1}f_1)(y,\eta )
(V_{\overline \phi _2}f_2)(y,\zeta -\eta )\, d\eta .
\notag
\end{align}
By Parseval's formula we get
\begin{align*}
G(y,\zeta )
&=
\int _{\rr d} (V_{\phi _1}f_1)(y,\eta )
(V_{\overline \phi _2}f_2)(y,\zeta -\eta )\, d\eta
\\[1ex]
&=
\int _{\rr d} \mascF (f_1\overline {\phi _1(\cdo -y)})(\eta )
\mascF (f_2{\phi _2(\cdo -y)})(\zeta -\eta )\, d\eta
\\[1ex]
&=
(\mascF (f_1\overline {\phi _1(\cdo -y)})\, ,\,
\mascF (\overline {f_2\, \phi _2(\cdo -y)}e^{i\scal \cdo \zeta}) )_{L^2(\rr d)}
\\[1ex]
&=
(f_1\overline {\phi _1(\cdo -y)}\, ,\,
\overline {f_2\, \phi _2(\cdo -y)}e^{i\scal \cdo \zeta} )_{L^2(\rr d)}
\\[1ex]
&=
\int _{\rr d}f_1(z)\overline {\phi _1(z-y)}
f_2(z)\phi _2(z-y)e^{-i\scal z\zeta}\, dz .
\end{align*}

\par

By inserting this into \eqref{Eq:ReformNaturalProd} and using
Fubini's theorem we get
\begin{align*}
((V_{\phi _1}f_1) &\natural _\phi (V_{\overline \phi _2}f_2))(x,\xi )
\\[1ex]
&=
(2\pi )^{-d}e^{-i\scal x\xi}
\iint _{\rr {2d}}
\overline {\widehat \phi (\zeta -\xi )}e^{-i\scal {z-x}\zeta}f_1(z)f_2(z)
H(z)
\, dzd\zeta ,
\intertext{where}
H(z) &=
\int _{\rr d}\phi _2(z-y)\overline {\phi _1(z-y)}\, dy
=
(\phi _2,\phi _1)_{L^2}.
\end{align*}
Hence, by evaluating the integral with respect to $\zeta$, and using
Fourier's inversion formula, we get
\begin{align*}
((V_{\phi _1}f_1) &\natural _\phi ((V_{\overline \phi _2}f_2)))(x,\xi )
\\[1ex]
&=
(2\pi )^{-\frac d2}e^{-i\scal x\xi}(\phi _2,\phi _1)_{L^2}
\int _{\rr {d}}\overline {\phi (z-x)}e^{i\scal {x-z}\xi}f_1(z)f_2(z)
\, dz
\\[1ex]
&=
(\phi _2,\phi _1)_{L^2} V_\phi (f_1f_2)(x,\xi ),
\end{align*}
which gives \eqref{eqWindowedProduct2}, and the result follows.
\end{proof}

\par

The formula \eqref{eqWindowedProduct2} is closely related to
\eqref{Eq:MultDefFormula1Mod}. In fact, the windows
$ \phi_j \in \Sigma _1 (\rr d)$,  $ j=0,1,2$, in
\eqref{Eq:MultDefFormula1Mod} should satisfy
the condition \eqref{Eq:MultDefFormula2A}, while
\eqref{eqWindowedProduct2}
is valid for arbitrary non-zero elements from
$ M^1 (\rr d)$.
For example, when $\phi =\phi _1=\phi _2$
and $\nm \phi{L^2 (\rr d)}=1$, then
\eqref{eqWindowedProduct2} reduces to
\begin{equation}
V_\phi (f_1 \cdot f_2) = (V_\phi f_1) \natural_\phi
(V_\phi f_2), \qquad f_1, f_2 \in  M^1 (\rr d),
\label{eqHomomorphism1}
\end{equation}
while \eqref{Eq:MultDefFormula1Mod} does not
allow such choice of windows.

\par

One of the main goals of \cite{DPT2022} are
extensions of the Gabor product
to some function spaces
$\mathcal{F}_j (\rr {2d})$, $j=0,1,2$, so that
$\natural_\phi$ maps $\mathcal{F}_1 \times \mathcal{F}_2$
into $\mathcal{F}_0$, with:
\begin{equation}
\|F_1 \natural_\phi F_2 \|_{\mathcal{F}_0} \leq C \|F_1\|_{\mathcal{F}_1} \|F_2\|_{\mathcal{F}_2}.
\label{eqIntro6}
\end{equation}
This can be considered as a phase space form of the Young
convolution inequality.

\par

Next we discuss continuity of the Gabor product on
certain spaces involving superpositions
of certain short-time Fourier transforms. In the
end we deduce properties similar to \cite[Theorem 29]{DPT2022}.
Instead of modulation spaces of the form
$M^{p,q}_{(\omega )}(\rr d)$, $p,q\in [1,\infty )$,
$\omega \in \mascP _E(\rr {2d})$,
here we consider modulation spaces of Wiener amalgam
types $W^{p,q}_{(\omega )}(\rr d)$, and allow the
"quasi-Banach" choice for Lebesgue parameters, i.{\,}e.
$p$ and $q$ are allowed to be smaller than one.

\par

Thus, in what follows we assume that
$p,q\in (0,\infty )$, $\omega \in
\mascP _E(\rr {2d})$ is $v$-moderate, and consider
$L^{p,q}_{*,(\omega )}(\rr {2d})$ spaces rather than
$L^{p,q}_{(\omega )}(\rr {2d})$ which are treated in
\cite{DPT2022}.

\par

We need some additional notation.
Let $ s > 1 $, $N\in \mathbf{N}$ be given, and let
$$
\mathscr{G}= \left
\{ \phi _n = \overline{\phi _n}\, ;~ n \in \mathbf{N}
\right \}
\subseteq  \Sigma_s (\rr d),
$$
be an orthonormal basis of $ L^2 (\rr d)$.
Then let
$\mathcal{V}_{\mathscr{G},\omega }^{(N),p,q}(\rr {2d})$
be the closure of
\begin{equation}
{\mathcal V}_{\mathscr{G}}^{(N)} (\rr {2d})
=
\left \{
\sum_{n=1}^N
V_{\phi _n}f_n~;~f_n \in  W^{1,1} _{(v)}(\rr d)
\right \}
\label{eqTheoremDensityWindows1}
\end{equation}
with respect to the $L^{p,q}_{*,(\omega )}(\rr {2d})$ norm.
In particular, if $N=1$, this reduces to
the closure $\overline{P_\phi }(L^{p,q}_{*,(\omega )}(\rr {2d}) )$
of
$$
P_\phi (L^{p,q}_{*,(\omega )}(\rr {2d}) )
= V_\phi (W^{1,1} _{(v)} (\rr d) )
$$
in the $L^{p,q}_{*,(\omega )}(\rr {2d})$ norm.

\par

By \cite[Theorem 26]{DPT2022}, it follows that for every  $ F \in
\mathcal{V}_{\mathscr{G},\omega}^{(N),p,q}(\rr {2d})$
there exist $f_n \in W_{(\omega )}^{p,q}(\rr {d})$,
$ n =1,2,\dots, N$, and such that
\begin{equation}
F = \sum_{n=1}^N V_{\phi _n}f_n~.
\label{eqTheoremDensityWindows2}
\end{equation}

\par

\begin{thm}\label{TheoremModulationSpace4}
Let $p_j, q_j\in (0,\infty )$ and
$\omega _j\in \mascP _E(\rr {2d})$, $j=0,1,2$,
be such that \eqref{Eq:LebExpHolderYoung2} and
\eqref{Eq:WeightCondHolderMod} hold, and let
$\phi \in  \Sigma_s (\rr d) $, $ s > 1 $.
Then the Gabor product $\natural_{\phi}$ from
$ {\mathcal V}_{\mathscr{G}}^{(N)}
\left(\rr {2d}\right) \times
{\mathcal V}_{\mathscr{G}}^{(N)} \left(\rr {2d}\right) $
to $W_{(v)}^{1,1}(\rr {2d})$, extends uniquely to
a continuous map from $\mathcal{V}_{\mathscr{G},\omega_1}^{(N),p_1,q_1}
(\rr {2d})
\times
\mathcal{V}_{\mathscr{G},\omega_2}^{(N),p_2,q_2}(\rr {2d})$ to
$\overline{P_\phi }(L^{p_0,q_0}_{*,(\omega_0 )}(\rr {2d}) )$,
and
\begin{equation}
\|F_1 \natural_{\phi} F_2 \|_{L_{*, (\omega_0)}^{p_0,q_0}}
\lesssim \|F_1 \|_{L_{*, (\omega_1)}^{p_1,q_1}} ~
\|F_2 \|_{L_{*, (\omega_2)}^{p_2,q_2}}~,
\label{eqModulationSpace8.1A}
\end{equation}
for all $F_j \in \mathcal{V}_{\mathscr{G},
\omega_j}^{(N),p_j,q_j}(\mathbf{R}^{2d}) $, $ j =1,2$.

\par

In particular, if $F_j =  V_\phi f_j $, $ j =1,2$, and
$\| \phi \| = 1$, then \eqref{eqModulationSpace8.1A}
reduces to
\begin{equation}
\| V_\phi f_1  \natural_{\phi}  V_\phi f_2  \|
_{L_{*, (\omega_0)}^{p_0,q_0}}=
\|  f_1  f_2  \|_{W_{(\omega _0)}^{p_0,q_0}}
\lesssim \|f_1 \|_{W_{(\omega _1)}^{p_1,q_1}}
~ \|f_2 \|_{W_{(\omega _2)}^{p_2,q_2}}~.
\label{eqModulationSpace8.1B}
\end{equation}
\end{thm}

\par

We omit the proof which is a slight modification of
the proof of Theorem 29 in \cite{DPT2022}.

\par

We end the paper by formally demonstrating how the Gabor
product arises in a phase space version of the cubic
Schr{\"o}dinger equation.
Consider the elliptic nonlinear Schr{\"o}dinger
equation (NLSE) given by
\begin{equation}
i \frac{\partial \psi}{\partial t} + \Delta \psi + \lambda |
\psi|^{2} \psi =0,
\label{eqNLSE2}
\end{equation}
subject to the initial condition:
\begin{equation*}
\psi (x,0)= \varphi (x).
\label{eqNLSE3}
\end{equation*}
Here $\lambda  = \pm 1 $ stands for an attracting
$(\lambda =+1)$ or repulsive $(\lambda =-1)$
power-law nonlinearity, and the Laplacian is given by
$$
\Delta = \sum_{j=1}^d \frac{\partial^2}{\partial x_j^2}.
$$
Thus we consider $ \psi = \phi (x,t) $ with
$ x \in \rr {d},$ and $ t$  in an open interval
$I \subseteq \mathbf{R}$.

\par

Using the following intertwining relations
\begin{equation*}
V_\phi (x_j \psi) = - D_{\xi _j} V_g (\psi),
\qquad
V_\phi ( D_{x_j} \psi ) = \left( \xi _j +D_{x_j} \right )
V_\phi \psi,
\label{eqIntertwiners}
\end{equation*}
$ j=1, \cdots, d$, and assuming that $\phi$ is a
real-valued window, we obtain upon application of
the STFT $V_\phi$ to (\ref{eqNLSE2}) that
\begin{equation}
i \frac{\partial F}{\partial t}
-
\sum_{j=1}^d \left(\xi _j + D_{x_j} \right)^2 F + \lambda
\widetilde F \natural_\phi F \natural_\phi F =0.
\label{eqSTFTCubicNLSE}
\end{equation}
Here, $D_{x_j}=-i \frac {\partial}{\partial x_j}$,
\begin{multline*}
F(x,\xi ,t) = V_\phi (\psi(\cdo ,t)) (x,\xi )
\\[1ex]
=
(2\pi )^{-\frac d2}
\int _{\rr d}
\psi(y,t) \overline{\phi (y-x)} e^{-i\scal y{\xi}}
\, dy, \quad x,\xi \in \rr d, \ t \in \mathbf{R},
\end{multline*}
and $\widetilde F $ is given by
\begin{equation}
\widetilde F(x, \xi ) = \overline{F (x, - \xi )}.
\label{eqInvolution}
\end{equation}

By considering \eqref{eqSTFTCubicNLSE}  the phase-space
formulation of the initial value problem may be
well-posed for more general initial distributions.
This means that the phase-space formulation "contains"
the solutions of the standard NLSE, but it is richer,
as it admits other solutions. We refer to
\cite{Dias1,Dias2,Dias3}, where phase-space
extensions are explored in several different contexts.

\par

Let us conclude by noticing that \eqref{eqSTFTCubicNLSE} contains
the triple product. Thus,  its qualitative analysis calls for a multilinear extension of Theorems \ref{Thm:MultMod2}
and  \ref{TheoremModulationSpace4}. Then the conditions \eqref{Eq:LebExpHolderYoung2} and
\eqref{Eq:WeightCondHolderMod} become more involved, see \cite{Toft26}.
Such analysis demands a more technical tools and arguments and goes beyond the scope of this survey article.

\appendix

\section{Some properties of Wiener
amalgam spaces}\label{App:A}

\par

There are convenient characterizations of
modulation spaces in the framework of
Gabor analysis.

Let $\omega _0\in \mascP _E(\rr d)$,
$\omega \in \mascP _E(\rr {2d})$, $p,q,r\in (0,\infty]$,
$Q_{d}=[0,1]^{d}$ be the unit cube,
and set for measurable $f$ on $\rr d$,
\begin{equation}\label{Eq:WienerQuasiNormsSimple}
\nm f{\sfW ^r(\omega _0,\ell ^{p})} \equiv \nm {a_0}{\ell ^{p}(\zz {d})}
\end{equation}
when
$$
a_0(j) \equiv \nm {f\cdot \omega _0}{L^r (j+Q_{d})},
\qquad
j\in \zz d,
$$
and measurable $F$ on $\rr {2d}$,
\begin{equation}\label{Eq:WienerQuasiNorms}
\nm F{\sfW ^r(\omega ,\ell ^{p,q})} \equiv \nm a{\ell ^{p,q}(\zz {2d})}
\quad \text{and}\quad
\nm F{\sfW (\omega ,\ell ^{p,q}_*)} \equiv \nm a{\ell ^{p,q}_*(\zz {2d})}
\end{equation}
when
$$
a(j,\iota ) \equiv \nm {F\cdot \omega}{L^r ((j,\iota )+Q_{2d})},
\qquad
j,\iota \in \zz d.
$$
The Wiener amalgam space
$$
\sfW ^r(\omega _0,\ell ^{p}) = \sfW ^r(\omega _0,\ell ^{p}(\zz {d}))
$$
consists of all measurable $f\in L^r _{loc}(\rr {d})$
such that $\nm F{\sfW ^r(\omega _0,\ell ^{p})}$ is finite, and
the Wiener amalgam spaces
$$
\sfW ^r(\omega ,\ell ^{p,q}) = \sfW ^r(\omega ,\ell ^{p,q}(\zz {2d}))
\quad \text{and}\quad
\sfW ^r(\omega ,\ell ^{p,q}_*) = \sfW ^r(\omega ,\ell ^{p,q}_*(\zz {2d}))
$$
consist of all measurable $F\in L^r _{loc}(\rr {2d})$
such that $\nm F{\sfW ^r(\omega ,\ell ^{p,q})}$ respectively
$\nm F{\sfW ^r(\omega ,\ell ^{p,q}_*)}$ are finite. We observe that
$\sfW ^r(\omega _0,\ell ^{p})$ is often denoted by
$W(L^r,L^p_{(\omega )})$ or $W(L^r,L^p_{(\omega )})$ in
the literature (see e.{\,}e. \cite{Fei3,FeiGro1,GaSa,Rau1}).

The topologies
are defined through their respectively quasi-norms in
\eqref{Eq:WienerQuasiNormsSimple} and
\eqref{Eq:WienerQuasiNorms}.
%
%
%
%
%
%
For conveniency we set
$$
\sfW (\omega ,\ell ^{p,q})=\sfW ^\infty (\omega ,\ell ^{p,q})
\quad \text{and}\quad
\sfW (\omega ,\ell ^{p,q}_*)=\sfW ^\infty (\omega ,\ell ^{p,q}_*),
$$
and if in addition $\omega =1$, we set
$$
\sfW (\ell ^{p,q})=\sfW (\omega ,\ell ^{p,q})
\quad \text{and}\quad
\sfW (\ell ^{p,q}_*)=\sfW (\omega ,\ell ^{p,q}_*).
$$

\par

Obviously, $\sfW ^r(\omega _0,\ell ^{p})$ and
$\sfW ^r(\omega ,\ell ^{p,q})$ increase with $p,q$,
decrease with $r$, and
\begin{alignat}{1}
\sfW (\omega ,\ell ^{p,q})
&\hookrightarrow
L^{p,q}_{(\omega )}(\rr {2d}) \cap \Sigma _1'(\rr {2d})
\hookrightarrow
L^{p,q}_{(\omega )}(\rr {2d})
\hookrightarrow
\sfW ^r(\omega ,\ell ^{p,q})
\intertext{and}
\nm \cdo {\sfW ^r(\omega ,\ell ^{p,q})}
&\le
\nm \cdo {L^{p,q}_{(\omega )}}
\le
\nm \cdo {\sfW (\omega ,\ell ^{p,q})},
\qquad
r\le\min (1,p,q).
\end{alignat}
On the other hand, for modulation spaces we have
\begin{equation}\label{Eq:ModWienerConn1}
f\in M^{p,q}_{(\omega )}(\rr d)
\quad \Leftrightarrow \quad
V_\phi f\in L^{p,q}_{(\omega )}(\rr {2d})
\quad \Leftrightarrow \quad
V_\phi f\in \sfW ^r(\omega ,\ell ^{p,q})
\end{equation}
with
\begin{equation}\label{Eq:ModWienerConn2}
\nm f{M^{p,q}_{(\omega )}} = \nm {V_\phi f}{L^{p,q}_{(\omega )}}
\asymp
\nm {V_\phi f}{\sfW ^r(\omega ,\ell ^{p,q})}.
\end{equation}
The same holds true with $W^{p,q}_{(\omega )}$, $L^{p,q}_{*,(\omega )}$
and $\sfW (\omega ,\ell ^{p,q}_*)$ in place of $M^{p,q}_{(\omega )}$,
$L^{p,q}_{(\omega )}$ and $\sfW (\omega ,\ell ^{p,q})$, respectively,
at each occurrence. (For $r=\infty$ , see \cite{Gro2} when $p,q\in [1,\infty ]$,
\cite{GaSa,Toft13} when $p,q\in (0,\infty ]$, and for $r\in (0,\infty ]$, see
\cite{Toft25}.)

\par

We have now the following result on the projection operator
$P_\phi$ in \eqref{Eq:ProjphiDef}
when acting on Wiener amalgam spaces.

\par

\begin{prop}\label{Prop:ProjMapModCont}
Let $p,q\in (0,\infty ]$ and
$\phi \in \mascS (\rr d)\setminus 0$. Then
$P_\phi$ from $\mascS '(\rr {2d})$ to $\mascS '(\rr {2d})$,
and $V_\phi ^*$ from $\mascS '(\rr {2d})$ to $\mascS '(\rr d)$
restrict to continuous mappings
\begin{align}
P_\phi : \sfW (\ell ^{p,q}(\zz {2d}))
&\to
V_\phi (M^{p,q}(\rr d)),
\label{Eq:ProjMapModCont1}
\\[1ex]
P_\phi : \sfW (\ell ^{p,q}_*(\zz {2d}))
&\to
V_\phi (W^{p,q}(\rr d)),
\label{Eq:ProjMapModCont2}
\\[1ex]
V_\phi ^* : \sfW (\ell ^{p,q} (\zz {2d}))
&\to
M^{p,q}(\rr d)
\label{Eq:ProjMapModCont3}
\intertext{and}
V_\phi ^* : \sfW (\ell ^{p,q}_*(\zz {2d}))
&\to
W^{p,q}(\rr d).
\label{Eq:ProjMapModCont4}
\end{align}
\end{prop}

\par

We refer to \cite[Proposition 3.6]{Toft26} for the proof of
Proposition \ref{Prop:ProjMapModCont} and
to \cite{FeiGro1,FeiGro3,Gro2,Rau1,Rau2,Toft26}
for some facts about the operators $P_\phi$ and $V_\phi ^*$,

\par

For $p,q\ge 1$, i.{\,}.e. the case when all spaces are Banach spaces,
proofs of Proposition \ref{Prop:ProjMapModCont} can be found
in e.{\,}g. \cite{Gro2} as well as in abstract forms in \cite{FeiGro1}.
In the general case when $p,q>0$, we refer to \cite{GaSa,Rau2},
since proofs of Proposition \ref{Prop:ProjMapModCont}
are essentially given there.

\par

\end{document}